\documentclass[]{interact}

\usepackage{epstopdf}
\usepackage[caption=false]{subfig}

\usepackage[numbers,sort&compress]{natbib}
\bibpunct[, ]{[}{]}{,}{n}{,}{,}
\renewcommand\bibfont{\fontsize{10}{12}\selectfont}

\usepackage{algpseudocode}
\theoremstyle{plain}
\newtheorem{theorem}{Theorem}[section]
\newtheorem{lemma}[theorem]{Lemma}

\theoremstyle{definition}
\newtheorem{definition}[theorem]{Definition}

\theoremstyle{remark}
\newtheorem{remark}{Remark}

\usepackage{graphicx}

\usepackage{chngcntr}

\usepackage{amssymb}
\usepackage{amsmath}
 \usepackage{bm}

 \usepackage[colorlinks=true,linkcolor=red, citecolor=blue, urlcolor=blue]{hyperref}%

 \usepackage{tikz}
 \usepackage{tkz-euclide}
 \usepackage{chngcntr}
\usepackage{verbatim}
 \usepackage{mathtools}
\usepackage{color}
\usepackage{xspace}
\usepackage{xparse}
 \usepackage{bbm}
\usepackage{enumerate}

\usepackage{systeme}

\usepackage{url}
\usepackage{algorithm}

\usepackage{makecell}
\usepackage{wrapfig,lipsum,booktabs}

\def\R{\mathbb{R}}
\newcommand{\E}{{\mathbb E}}

\def\X{\mathcal X}

\def\R{\mathbb R}
\def\E{\mathbb E}

\def\PP{\mathbb P}

\def\lm{\lambda}

\def\e{\varepsilon}
\def\la{\langle}
\def\ra{\rangle}

\def\y{\mathbf{y}}

\def\x{\mathbf{x}}
\def\b{\mathbf{b}}

\def\one{{\mathbf 1}}

\def\avg#1{{#1}}
\def\ag#1{{#1}} 
\def\dm#1{{#1}} 
\newcommand{\ddd}[1]{\textcolor{cyan}{#1}} 

\usepackage[percent]{overpic}

\begin{document}
\articletype{ARTICLE TEMPLATE}

\title{Stochastic Approximation versus Sample Average Approximation for  Wasserstein barycenters
}

\author{
\name{Darina Dvinskikh\textsuperscript{a,b,c}\thanks{CONTACT Darina Dvinskikh. Email: dviny.d@yandex.ru} }
\affil{
\textsuperscript{a}Weierstrass Institute for Applied Analysis and Stochastics, Berlin, Germany;\\
\textsuperscript{b}Moscow Institute of Physics and Technology, Dolgoprudny, Russia;\\
\textsuperscript{c}Institute for Information Transmission Problems, Moscow, Russia.
}
}

\maketitle

\begin{abstract}

In the machine learning and optimization community, there are two main approaches for the convex risk minimization problem, namely, the Stochastic Approximation (SA) and the Sample Average Approximation (SAA).
In terms of oracle complexity (required number of stochastic gradient evaluations), both approaches are considered equivalent on average (up to a logarithmic factor). The total complexity depends on the specific problem, however, starting from work \cite{nemirovski2009robust} it was generally accepted that the SA is better than the SAA.
 We show that for the Wasserstein barycenter problem this superiority can be inverted. We provide a detailed comparison by stating the complexity bounds for  the SA and the SAA implementations calculating barycenters  defined with respect to optimal transport distances and entropy-regularized optimal transport distances. As a byproduct, we also construct  confidence intervals for the barycenter  defined with respect to  entropy-regularized optimal transport distances in the $\ell_2$-norm.

The preliminary results are derived for a general convex optimization problem given by the expectation in order to have other applications besides the Wasserstein barycenter problem.

\end{abstract}

\section{Introduction}

In this paper, we consider the problem of finding  a barycenter of discrete random probability measures generated by  a  distribution. We refer to 
optimal transport (OT) metrics which provides a successful framework to compare objects that can be modeled  as probability measures (images, videos, texts and etc.).  Transport based distances have gained popularity in various fields such as statistics \cite{ebert2017construction,bigot2012consistent}, unsupervised learning \cite{arjovsky2017wasserstein}, signal and image analysis \cite{thorpe2017transportation},  computer vision \cite{rubner1998metric}, text classification \cite{kusner2015word}, economics and finance \cite{rachev2011probability} and medical imaging \cite{wang2010optimal,gramfort2015fast}. 
Moreover, a lot of 
 statistical results are known about optimal transport  distances \cite{sommerfeld2018inference,weed2019sharp,klatt2020empirical}.

The  success of optimal transport  led to an increasing interest in  {Wasserstein barycenters} (WB's).
Wasserstein barycenters are used in Bayesian computations \cite{srivastava2015wasp}, texture mixing \cite{rabin2011wasserstein}, clustering  ($k$-means for probability measures) \cite{del2019robust}, shape interpolation and color transferring \cite{Solomon2015},  statistical estimation of template models \cite{boissard2015distribution} and neuroimaging \cite{gramfort2015fast}.
For discrete random probability measures from probability simplex $\Delta_n$ ($n$ is the size of support) with distribution $\PP$, a Wasserstein barycenter is introduced through a notion of  Fr\'{e}chet mean \cite{frechet1948elements}
\begin{equation}\label{def:Freche_population}
  \min_{p\in  \Delta_n}\E_{q \sim \PP} W(p,q).
\end{equation}
If a solution of \eqref{def:Freche_population} exists and is unique, then it is referred to as the population barycenter for distribution $\PP$. Here  $W(p,q)$ is 
 optimal transport metrics between measures $p$ and  $q$  
\begin{equation}\label{def:optimaltransport}
  W(p,q) =   \min_{\pi \in U(p,q)} \la C, \pi \ra,
\end{equation}
where $C  \in \R^{n\times n}_+$ is a symmetric transportation cost matrix and $U(p,q) \triangleq\{ \pi\in \R^{n\times n}_+: \pi \one =p, \pi^T \one = q\}$ is transport polytope.\footnote{
When for $\rho\geq 1$, $C_{ij} =\mathtt d(x_i, x_j)^\rho$  in \eqref{def:optimaltransport}, where $\mathtt d(x_i, x_j)$ is a distance on support points $x_i, x_j$, then $W(p,q)^{1/\rho}$ is known as the $\rho$-Wasserstein distance.
Nevertheless, all the results of this thesis  are based only on the assumptions that the  matrix $C \in \R_+^{n\times n}$ is symmetric and non-negative. Thus, optimal transport  problem defined in \eqref{def:optimaltransport} is    a more general  than the  Wasserstein distances.}

In \cite{cuturi2013sinkhorn}, the  entropic regularization of optimal transport problem \eqref{def:optimaltransport}  was proposed to   improve its statistical properties   \cite{klatt2020empirical,bigot2019central} and to  reduce the computational complexity  from $\tilde O(n^3)$  ($n$ is the size of the support of the measures) to $n^2\min\{\tilde O\left(\frac{1}{\e} \right), \tilde O\left(\sqrt{n} \right) \} $ arithmetic operations\footnote{The estimate $n^2\min\{\tilde O\left(\frac{1}{\e} \right), \tilde O\left(\sqrt{n} \right) \}$ is the best known theoretical  estimate for solving OT problem \cite{blanchet2018towards,jambulapati2019direct,lee2014path,quanrud2018approximating}. The best  known practical  estimates are $\sqrt{n}$ times worse (see \cite{guminov2019accelerated} and references therein).}
\begin{align}\label{eq:wass_distance_regul2222}
W_\gamma (p,q) &\triangleq \min_{\pi \in  U(p,q)} \left\lbrace \left\langle  C,\pi\right\rangle - \gamma E(\pi)\right\rbrace.
\end{align}
Here $\gamma>0$ and $E(\pi) \triangleq -\la \pi,\log \pi \ra $ is the  entropy. Since $E(\pi)$ is 1-strongly concave on $\Delta_{n^2}$ in the  $\ell_1$-norm, the objective in \eqref{eq:wass_distance_regul2222} is $\gamma$-strongly convex  with respect to $\pi$ in the $\ell_1$-norm on $\Delta_{n^2}$, and hence, problem \eqref{eq:wass_distance_regul2222} has a unique optimal solution. Moreover, $W_\gamma (p,q)$ is $\gamma$-strongly convex with respect to $p$ in the $\ell_2$-norm on $\Delta_n$ \citep[Theorem 3.4]{bigot2019data}.
Another particular advantage of the entropy-regularized  optimal transport \eqref{eq:wass_distance_regul2222} is a closed-form representation for its dual function~\cite{agueh2011barycenters,cuturi2016smoothed} defined by the Fenchel--Legendre transform of $W_\gamma(p,q)$ as a function of $p$
\begin{align*}
     W_{\gamma, q}^*(u) &=  \max_{ p \in \Delta_n}\left\{ \la u, p \ra - W_{\gamma}(p, q) \right\} =  \gamma\left(E(q) + \left\langle q, \log (K \beta) \right\rangle  \right).
\end{align*}
  where  $\beta = \exp( {u}/{\gamma}) $, \mbox{$K = \exp( {-C}/{\gamma }) $} and functions $\log$ or $\exp$ are  applied element-wise.
Hence, the gradient of dual function $W_{\gamma, q}^*(u)$ is also represented in  a closed-form \cite{cuturi2016smoothed}
\begin{equation*}
\nabla W^*_{\gamma,q} (u)
= \beta \odot \left(K \cdot {q}/({K \beta}) \right) \in \Delta_n,
\end{equation*}
where symbols $\odot$ and $/$ stand for the element-wise product and element-wise division respectively.

\textbf{Background on the SA and the SAA and Convergence Rates.}
Let us consider a general  stochastic convex minimization problem
\begin{equation}\label{eq:gener_risk_min}
\min_{x\in X \subseteq \mathbb{R}^n} F(x) \triangleq \E f(x,\xi),    
\end{equation}
where function $f$ is convex in $x$ ($x\in X,$ $X$ is a convex set), and $\E f(x,\xi)$ is the expectation of $f$ with respect to $\xi \in \Xi$.
Such kind of problems arise in many applications of data science 
\cite{shalev2014understanding,shapiro2014lectures} (e.g., risk minimization) and mathematical statistics \cite{spokoiny2012parametric} (e.g., maximum likelihood estimation).  There are two competing  approaches based on Monte Carlo sampling techniques to solve \eqref{eq:gener_risk_min}: the Stochastic Approximation (SA) \cite{robbins1951stochastic} and the Sample Average Approximation (SAA).  
The SAA approach  replaces the objective in problem \eqref{eq:gener_risk_min} with its sample average approximation (SAA) problem
\begin{equation}\label{eq:empir_risk_min}
 \min_{x\in X} \hat{F}(x) \triangleq \frac{1}{m}\sum_{i=1}^m f(x,\xi_i), 
\end{equation}
where  $\xi_1, \xi_2,...,\xi_m$ are the realizations of a random variable  $\xi$. The number of realizations $m$ is adjusted by the desired precision.
The total working time of both approaches  to solve  problem \eqref{eq:gener_risk_min} with the average precision $\e$ in the non-optimality gap in term of the objective function (i.e., to find  $x^N$ such that $\E F(x^N) - \min\limits_{x\in X } F(x)\leq \e$),
depends on the specific problem. However, it was generally accepted \cite{nemirovski2009robust} that the SA approach is better than the SAA approach.
  Stochastic gradient (mirror) descent, an implementation of the SA approach \cite{juditsky2012first-order}, gives the following estimation for the number of iterations (that is equivalent to the sample size of $\xi_1, \xi_2, \xi_3,...,\xi_m$)
\begin{equation}\label{eq:SNSm}
m = O\left(\frac{M^2R^2}{\e^2}\right).
\end{equation}
Here we considered the minimal assumptions (non-smoothness) for the objective $f(x,\xi)$
\begin{equation}\label{M}
\|\nabla f(x,\xi)\|_2^2\le M^2, \quad \forall x \in X, \xi \in \Xi.
\end{equation}
Whereas, the application of the SAA approach requires  the following sample size \cite{shapiro2005complexity}
\[m = \widetilde{O}\left(\frac{n M^2R^2}{\e^2}\right),\]
that is $n$ times more ($n$ is the problem’s dimension) than the sample size in the SA approach. This estimate was obtained under the assumptions that problem \eqref{eq:empir_risk_min} is solved exactly. This is one of the main drawback of the SAA approach. However, if the objective $f(x,\xi)$ is $\lm$-strongly convex in $x$, the sample sizes are equal  up to logarithmic terms
\[
m = O\left(\frac{M^2}{\lm \e}\right).
\]
Moreover, in this case, for the SAA approach, it suffices to solve problem \eqref{eq:empir_risk_min}  with accuracy \cite{shalev2009stochastic}
 \begin{equation}\label{eq:aux_e_quad}
   \e' = O\left(\frac{\e^2\lm}{M^2}\right).
 \end{equation}
 Therefore, to eliminate the linear dependence on  $n$ in the SAA approach for a non-strongly convex objective, regularization $\lm =\frac{\e}{R^2}$ should be used \cite{shalev2009stochastic}.

Let us suppose that $f(x,\xi)$ in  \eqref{eq:gener_risk_min} is convex but non-strongly convex in $x$ (possibly, $\lm$-strongly convex but with very small $\lm \ll \frac{\e}{R^2}$). Here  $R = \|x^1 - x^*\|_2$ is the Euclidean distance between starting point $x^1$ and the solution $x^*$ of \eqref{eq:gener_risk_min} which corresponds to the minimum of this norm (if the solution is not the only one). Then,  the problem \eqref{eq:gener_risk_min} can be replaced by 
\begin{equation}\label{eq:gener_risk_min_reg}
\min_{x\in X }  \E f(x,\xi) + \frac{\e}{2R^2}\|x - x^1\|_2^2.
\end{equation}
 The empirical counterpart of  \eqref{eq:gener_risk_min_reg}  is 
\begin{equation}\label{eq:empir_risk_min_reg}
 \min_{x\in X} \frac{1}{m}\sum_{i=1}^m f(x,\xi_i) + \frac{\e}{2R^2}\|x - x^1\|_2^2,   
\end{equation}
where the sample size $m$ is defined in   \eqref{eq:SNSm}
Thus, in the case of non-strongly objective, a regularization equates the sample size of both approaches.


\subsection{Contribution and Related  Work}

 \hspace{0.25cm} \textbf{The SA and the SAA approaches}. 
 This paper is inspired by the work \cite{nemirovski2009robust}, where it is stated that the SA approach outperforms the SAA approach for a certain class of convex stochastic problems. Our aim is to show that for the Wasserstein barycenter problem this superiority can be inverted. We provide a detailed comparison by stating the complexity bounds for implementations of the SA and the SAA approaches for the Wasserstein barycenter problem.  As a byproduct, we also construct a confidence interval for the barycenter  defined w.r.t. entropy-regularized OT.

 \textbf{Sample size.} We also estimate the sample size of measures to calculate an approximation for Fr\'{e}chet mean of a probability distribution with a given precision.

 \textbf{Consistency and rates of convergence.}
The consistency of empirical barycenter as an estimator of  true Wasserstein barycenter (defined by the notion of Fr\'{e}chet mean) as the number of measures  tends to infinity  was studied in many papers, e.g, \cite{LeGouic2017,panaretos2019statistical,LeGouic2017,Bigot2012a,rios2018bayesian}, 
under  some conditions for the process generated the measures.
Moreover, the authors of \cite{boissard2015distribution} provide the rate of this convergence but under  restrictive assumption on the process (it must be from 
{admissible family of deformations}, i.e., it is  a gradient of a convex  function).   Without any assumptions on generating process, the rate of convergence was obtained in \cite{bigot2018upper}, however, only for measures with one-dimensional support. For some specific types of metrics and measures, the rates of convergence  were also provided in works \cite{chewi2020gradient,gouic2019fast,kroshnin2019statistical}. Our results were obtained  under the condition of discreteness of the measures. We can always  achieve this condition  through additional preprocessing (discretization of measures).

 

 \textbf{Penalization of  barycenter problem.} 
 For a general convex (but not strongly convex) optimization problem,  empirical minimization may fail  in offline approach despite the guaranteed success of an online approach if no regularization was introduced \cite{shalev2009stochastic}. 
 The limitations of the SAA approach for non-strongly convex case are also discussed in \cite{guigues2017non-asymptotic,shapiro2005complexity}. 
Our contribution  includes  introducing a new regularization for population Wasserstein barycenter problem that improves the complexity bounds for standard  penalty (squared norm penalty) \cite{shalev2009stochastic}.  This regularization relies on the  Bregman divergence from \cite{ben-tal2001lectures}.





\subsection{Preliminaries}

\noindent\textbf{Notations}. 
Let  $\Delta_n  = \{ a \in \mathbb{R}_+^n  \mid \sum_{l=1}^n a_l =1 \}$ be the probability simplex. 
Then we refer 
to the $j$-th component of  vector $x_i$ as $[x_i]_j$.
The notation $[n]$ means $1,2,...,n$.
For two vectors $x,y$ of the same size, denotations $x/y$ and $x \odot y$  stand for the element-wise product and  element-wise division respectively. 
When functions, such as $log$ or $exp$, are used on vectors, they are always applied element-wise.
For some norm $\|\cdot\|$ on space $\X$, we define the dual norm $\|\cdot\|_*$ on the dual space $\X^*$ in a usual way $ \|s\|_{*} = \max\limits_{x\in \X} \{ \la x,s \ra : \|x\| \leq 1 \} $. 
 We denote by $I_n$ the identity matrix, and by $0_{n\times n}$ we denote zeros matrix.
For a positive semi-definite matrix $A$ we denote its  smallest positive eigenvalue by $\lm^{+}_{\min}(A)$.
We  use denotation $ O(\cdot)$ when we want to indicate the complexity hiding constants, to hide also  logarithms, we use denotation $\widetilde O(\cdot)$.


\begin{definition}
A function $f(x,\xi):X \times \Xi \rightarrow \R$ is $M$-Lipschitz continious in $x$ w.r.t. a norm $\|\cdot\|$ if it satisfies
\[{|}f(x,\xi)-f(y,\xi){|}\leq M\|x-y\|, \qquad \forall x,y \in X,~ \forall \xi \in \Xi.\]
\end{definition}

\begin{definition}
A function $f:X\times \Xi \rightarrow \R$ is $\gamma$-strongly convex in $x$ w.r.t. a norm $\|\cdot\|$ if it is continuously differentiable and it satisfies
\[f(x, \xi)-f(y, \xi)- \la\nabla f(y, \xi), x-y\ra\geq \frac{\gamma}{2}\|x-y\|^2, \qquad \forall x,y \in  X,~ \forall \xi \in \Xi.\]
\end{definition}
\begin{definition}
The Fenchel--Legendre conjugate for a function $f:(X, \Xi) \rightarrow \R$ w.r.t. $x$ is \[f^*(u,\xi) \triangleq \sup_{x \in X}\{\la x,u\ra - f(x,\xi)\}, \qquad  \forall \xi \in \Xi.\]
\end{definition}

\subsection{Paper organization}
The structure of the paper is the following. 
In Section \ref{ch:population} 
we give a background on  
the SA and the SAA approaches and derive  preliminary results.
 Section \ref{sec:pen_bar} presents the comparison of the SA and the SAA approaches for the problem of Wasserstein barycenter defined w.r.t. regularized optimal transport distances. Finally, Section \ref{sec:unreg} gives the comparison of the SA and the SAA approaches for the problem of Wasserstein barycenter defined w.r.t. (unregularized) optimal transport distances.

\section{Strongly Convex Optimization Problem}
\label{ch:population}

We start with preliminary results stated for a general stochastic strongly convex  optimization problem  of form
\begin{equation}\label{eq:gener_risk_min_conv}
\min_{x\in X \subseteq \mathbb{R}^n} F(x) \triangleq \E f(x,\xi),
\end{equation}
where $f(x,\xi)$ is $\gamma$-strongly convex with respect to $x$. Let us define   $ x^* = \arg\min\limits_{x\in X} {F}(x)$.

\subsection{The SA Approach: Stochastic Gradient Descent }

The classical SA algorithm for problem  \eqref{eq:gener_risk_min_conv} is presented by  stochastic gradient descent (SGD) method. We consider the SGD with inexect oracle  given by $g_\delta(x,\xi)$ such that
\begin{equation}\label{eq:gen_delta}
\forall x \in X, \xi \in \Xi, \qquad
\|\nabla f(x,\xi) - g_\delta(x,\xi)\|_2 \leq \delta.
\end{equation}
Then the iterative  formula of SGD can be written as ($k=1,2,...,N.$)
\begin{equation}\label{SA:implement_simple}
  x^{k+1} = \Pi_{X}\left(x^k - \eta_{k} g_\delta(x^k,\xi^k) \right). 
\end{equation}
Here $x^1 \in X$ is  starting point, $\Pi_X$ is the projection onto $X$, $\eta_k$ is a stepsize.
For a $\gamma$-strongly convex $f(x,\xi)$ in $x$, stepsize $\eta_k$ can be taken as $\frac{1}{\gamma k}$ to obtain  optimal rate $O(\frac{1}{\gamma N})$.

A good indicator of the success of an algorithm is the \textit{regret}
 \[Reg_N \triangleq \sum_{k=1}^{N} \left(f( x^k,  \xi^k) -  f(x^*,  \xi^k)\right).\]
It measures the value of the difference between a made decision and the optimal decision on all the rounds.
The work \cite{kakade2009generalization} gives a  bound on the excess risk of the output of an online algorithm in terms of the average regret.

\begin{theorem}\citep[Theorem 2]{kakade2009generalization} \label{Th:kakade2009generalization}
Let  $f:X\times \Xi \rightarrow [0,B]$ be $\gamma$-strongly convex and $M$-Lipschitz w.r.t. $x$. Let $\tilde x^N \triangleq \frac{1}{N}\sum_{k=1}^{N}x^k $ 
be the average of online vectors $x^1, x^2,...,x^N$.
Then with probability at least $1-4\beta\log N$
\[ F(\tilde x^N) - F(x^*) \leq \frac{Reg_N }{N} + 4\sqrt{ \frac{M^2\log(1/\beta)}{\gamma}}\frac{\sqrt{Reg_N}}{N} + \max\left\{ \frac{16M^2}{\gamma},6B \right\}\frac{\log(1/\beta)}{N}.  \]
\end{theorem}
For the update rule \eqref{SA:implement_simple} with $\eta_k = \frac{1}{\gamma k}$, this theorem can be specify as follows.

\begin{theorem}\label{Th:contract_gener}
Let  $f:X\times \Xi \rightarrow [0,B]$ be $\gamma$-strongly convex and $M$-Lipschitz w.r.t. $x$. Let $\tilde x^N \triangleq \frac{1}{N}\sum_{k=1}^{N}x^k $ be the average of outputs generated by iterative formula \eqref{SA:implement_simple} with $\eta_k = \frac{1}{\gamma k}$. Then,  with probability  
at least $1-\alpha$  the following holds
\begin{align*}
    F(\tilde x^N) - F(x^*) 
    &\leq  \frac{3\delta D }{2}   + \frac{3(M^2+\delta ^2)}{N\gamma }(1+\log N)  \notag\\
    &+ \max\left\{ \frac{18M^2}{\gamma},6B + \frac{2M^2}{\gamma} \right\}\frac{\log(4\log N/\alpha)}{N}.
\end{align*}
where $D =\max\limits_{x',x'' \in X}\|x'-x''\|_2 $ and $\delta$ is defined by \eqref{eq:gen_delta}.
\end{theorem}

\begin{proof}  
The proof mainly relies on  Theorem \ref{Th:kakade2009generalization} and estimating the regret for iterative formula \eqref{SA:implement_simple} with $\eta_k = \frac{1}{\gamma k}$.

From $\gamma$-strongly convexity in  $x$ of $f(x,\xi)$,  it follows for $x^k, x^* \in X$
\begin{equation*}
   f(x^*, \xi^k) \geq f(x^k,  \xi^k) + \la \nabla f(x^k, \xi^k), x^*-x^k\ra +\frac{\gamma}{2}\|x^*-x^k\|_2. 
\end{equation*}
Adding and subtracting the term  $\la g_\delta(x^k, \xi^k), x^*-x^k\ra$ we get using Cauchy–Schwarz inequality and \eqref{eq:gen_delta} 
\begin{align}\label{str_conv_W1}
   f(x^*,  \xi^k) &\geq f(x^k,  \xi^k) + \la g_\delta(x^k,\xi^k), x^*-x^k\ra +\frac{\gamma}{2}\|x^*-x^k\|_2   \notag \\
    &+ \la \nabla f(x^k,  \xi^k) -g_\delta(x^k,\xi^k), x^*-x^k\ra \notag\\
    &\geq f(x^k,  \xi^k) + \la g_\delta(x^k,\xi^k), x^*-x^k\ra + 
    \frac{\gamma}{2}\|x^*-x^k\|_2 + \delta\|x^*-x^k\|_2.
\end{align}
From the update rule \eqref{SA:implement_simple} for $x^{k+1}$  we have
\begin{align*}
   \|x^{k+1} - x^*\|_2 &= \|\Pi_{X}(x^k - \eta_{k} g_\delta(x^k,\xi^k)) - x^*\|_2 \notag\\
   &\leq \|x^k - \eta_{k} g_\delta(x^k,\xi^k) - x^*\|_2 \notag\\
   &\leq \|x^k  - x^*\|_2^2 + \eta_{k}^2\| g_\delta(x^k,\xi^k)\|_2^2 -2\eta_{k}\la g_\delta(x^k,\xi^k), x^k  - x^*\ra.
\end{align*}
From this it follows
\begin{equation*}
    \la g_\delta(x^k,\xi^k), x^k  - x^*\ra \leq \frac{1}{2\eta_{k}}(\|x^k-x^*\|^2_2 - \|x^{k+1} -x^*\|^2_2) + \frac{\eta_{k}}{2}\| g_\delta(x^k,\xi^k)\|_2^2.
\end{equation*}
Together with \eqref{str_conv_W1} we get
\begin{align*}
     f(x^k,  \xi^k) -  f(x^*,  \xi^k) &\leq  \frac{1}{2\eta_{k}}(\|x^k-x^*\|^2_2 - \|x^{k+1} -x^*\|^2_2)   \notag \\
     &-\left(\frac{\gamma}{2}+\delta\right)\|x^*-x^k\|_2+ \frac{\eta_{k}^2}{2}\| g_\delta(x^k,\xi^k)\|_2^2.
\end{align*}
Summing this from 1 to $N$,    we get using $\eta_k = \frac{1}{\gamma k}$
\begin{align}\label{eq:eq123}
 \sum_{k=1}^{N}f( x^k,  \xi^k) -  f(x^*,  \xi^k) &\leq 
     \frac{1}{2}\sum_{k=1}^{N}\left(\frac{1}{\eta_k} -  \frac{1}{\eta_{k-1}} + {\gamma} +\delta \right)\|x^*-x^k\|_2 \notag\\ &\hspace{-1cm}+\frac{1}{2}\sum_{k=1}^{N}{\eta_{k}}\| g_\delta (x^k,\xi^k)\|_2^2 \notag\\
     &\hspace{-1cm}\leq \frac{\delta}{2}\sum_{k=1}^{N}\|x^*-x^k\|_2 +\frac{1}{2}\sum_{k=1}^{N}{\eta_{k}}\| g_\delta (x^k,\xi^k)\|_2^2.
 \end{align}
 From Lipschitz continuity of $f(x,\xi)$ w.r.t. to $x$ it follows that $\|\nabla f(x,\xi)\|_2\leq M$ for all $x \in X, \xi \in \Xi$. Thus, using that for all $a,b, ~ (a+b)^2\leq 2a^2+2b^2$ it follows
 \[
 \|g_\delta(x,\xi)\|^2_2 \leq 2\|\nabla f(x,\xi)\|^2_2 + 2\delta^2 = 2M^2 + 2\delta^2
 \]
 From this and \eqref{eq:eq123} we bound the regret as follows
 \begin{align}\label{eq:reg123}
 Reg_N \triangleq \sum_{k=1}^{N}f( x^k,  \xi^k) -  f(x^*,  \xi^k)   &\leq  \frac{\delta}{2}\sum_{k=1}^{ N}\|p^*-p^k\|_2 + (M^2 +\delta^2)\sum_{k=1}^{ N} \frac{1}{\gamma k}\notag\\
     &\leq \frac{1}{2}  \delta D N + \frac{M^2+\delta ^2}{\gamma }(1+\log N). 
 \end{align}
Here the last bound takes place due to the sum of harmonic series.
Then for \eqref{eq:reg123} we can use Theorem \ref{Th:kakade2009generalization}. Firstly, we simplify it rearranging the terms using that $\sqrt{ab} \leq \frac{a+b}{2}$ 
\begin{align*}
    F(\tilde x^N) - F(x^*) 
    &\leq \frac{Reg_N }{N} + 4\sqrt{ \frac{M^2\log(1/\beta)}{N\gamma}}\sqrt{\frac{Reg_N}{N}} + \max\left\{ \frac{16M^2}{\gamma},6B \right\}\frac{\log(1/\beta)}{N}
    \notag\\
    &\leq \frac{3 Reg_N }{N} + \frac{2 M^2\log(1/\beta)}{N\gamma} + \max\left\{ \frac{16M^2}{\gamma},6B \right\}\frac{\log(1/\beta)}{N}\notag\\
    &= \frac{3 Reg_N }{N}  + \max\left\{ \frac{18M^2}{\gamma},6B + \frac{2M^2}{\gamma} \right\}\frac{\log(1/\beta)}{N}.
\end{align*}
Then we substitute \eqref{eq:reg123} in this inequality and making change $\alpha = 4\beta\log N$ and get with probability at least $ 1-\alpha$
\begin{align*}
    F(\tilde x^N) - F(x^*) 
    &\leq  \frac{3\delta D }{2}   + \frac{3(M^2+\delta ^2)}{N\gamma }(1+\log N)  \notag\\
    &+ \max\left\{ \frac{18M^2}{\gamma},6B + \frac{2M^2}{\gamma} \right\}\frac{\log(4\log N/\alpha)}{N}.
\end{align*}
\end{proof}

\subsection{Preliminaries on the SAA Approach }
The SAA approach replaces the objective in  \eqref{eq:gener_risk_min_conv} with its sample average 
\begin{equation}\label{eq:empir_risk_min_conv}
 \min_{x\in X } \hat{F}(x) \triangleq \frac{1}{m}\sum_{i=1}^m f(x,\xi_i),  
\end{equation}
where each $f(x,\xi_i)$ is $\gamma$-strongly convex in $x$.
Let us define the empirical minimizer of \eqref{eq:empir_risk_min_conv} $\hat x^* = \arg\min\limits_{x\in X} \hat{F}(x)$, and   $\hat x_{\e'}$  such that
\begin{equation}\label{eq:fidelity}
\hat{F}(\hat x_{\e'}) - \hat{F}(\hat x^*)  \leq \e'.
\end{equation}
The next theorem  gives a  bound on the excess risk for problem \eqref{eq:empir_risk_min_conv}
in the SAA approach.
\begin{theorem}\label{Th:contractSAA}
Let  $f:X\times \Xi \rightarrow [0,B]$ be $\gamma$-strongly convex and $M$-Lipschitz w.r.t. $x$ in the $\ell_2$-norm. Let $\hat x_{\e'}$ satisfies  \eqref{eq:fidelity}  with precision $ \e' $. 
Then,  with probability at least $1-\alpha$  we have
\begin{align*}
F( \hat x_{\e'}) -  F(x^*)
   &\leq \sqrt{\frac{2M^2}{\gamma}\e'}  +\frac{4M^2}{\alpha\gamma m}. 
\end{align*}
Let $\e' = O \left(\frac{\gamma\e^2}{M^2}  \right)$ and $m = O\left( \frac{M^2}{\alpha \gamma \e} \right)$. Then, with probability  at least $1-\alpha$  the following holds
\[F( \hat x_{\e'}) -  F(x^*)\leq \e \quad \text{and} \quad \|\hat x_{\e'} - x^*\|_2 \leq \sqrt{2\e/\gamma}.\]
\end{theorem}
The proof of this theorem mainly relies on  the following theorem.
\begin{theorem}\citep[Theorem 6]{shalev2009stochastic}\label{Th:shalev2009stochastic}
Let $f(x,\xi)$ be $\gamma$-strongly convex and $M$-Lipschitz  w.r.t. $x$ in the $\ell_2$-norm.  Then, with probability at least $1-\alpha$ the following holds
\[
F(\hat x^* ) - F(x^*) \leq \frac{4M^2}{\alpha \gamma m},
\]
where $m$ is the sample size.
\end{theorem} 
\begin{proof}[Proof of Theorem \ref{Th:contractSAA}]

For any $x\in X$, the following holds
\begin{equation}\label{eq:exp_der}
 F(x) -  F(x^*) =    F(x) -   F(\hat x^*)+ F( \hat x^*)- F(x^*). 
\end{equation}
From  Theorem \ref{Th:shalev2009stochastic}  with probability  at least $1 - \alpha$  the following holds
\begin{equation*}
F( \hat x^*) -  F(x^*) \leq  \frac{4M^2}{\alpha \gamma m}.
\end{equation*}
Then from this and \eqref{eq:exp_der}  we have  with probability  at least $1 - \alpha$
\begin{equation}\label{eq:eqtosub}
 F(x) -  F(x^*) \leq F(x) -   F(\hat x^*)+  \frac{4M^2}{\alpha \gamma m}. 
\end{equation}
From Lipschitz continuity of $ f(x,\xi)$ it follows, that for ant $x\in X, \xi\in \Xi$ the following holds
\begin{equation*}
| f(x,\xi) - f(\hat x^*,\xi)| \leq M\|x-\hat x^*\|_2.
\end{equation*}
Taking the expectation of this inequality w.r.t. $\xi$ we get 
\begin{equation*}
   \E | f(x,\xi) - f(\hat x^*,\xi)| \leq M\|x-\hat x^*\|_2.
\end{equation*}
Then we use Jensen's inequality ($g\left (\E(Y)\right) \leq \E g(Y) $) for the expectation,  convex function $g$ and  a random variable $Y$. Since the module is a convex function we get
\begin{equation*}
 | \E f(x,\xi) - \E f(\hat x^*,\xi)| =| F(x) -  F(\hat x^*)| \leq   \E | f(x,\xi) - f(\hat x^*,\xi)| \leq M\|x-\hat x^*\|_2.
\end{equation*}
Thus, we have 
\begin{equation}\label{eq:Lipsch}
   | F(x) -  F(\hat x^*)| \leq M\|x-\hat x^*\|_2.
\end{equation}
From strong convexity of $ f( x, \xi)$ in $x$, it follows that the average of $f(x,\xi_i)$'s, that is $\hat F(x)$, is also $\gamma$-strongly convex in $x$. Thus we get for any $x \in X, \xi \in \Xi$
\begin{equation}\label{eq:str}
\|x-\hat x^*\|_2 \leq \sqrt{\frac{2}{\gamma } (\hat F(x) - \hat F(\hat x^*))}.
\end{equation}
By using \eqref{eq:Lipsch} and \eqref{eq:str}  \label{eq:th6} and taking $x=\hat x_{\e'}$ in \eqref{eq:eqtosub},  we get the first statement of the theorem
\begin{align}\label{eq_to_prove_2}
 F( \hat x_{\e'}) -  F(x^*)
  &\leq   \sqrt{\frac{2M^2}{\gamma }(\hat F( \hat x_{\e'}) -  \hat F(\hat x^*))} +\frac{4M^2}{\alpha\gamma m} \leq \sqrt{\frac{2M^2}{\gamma}\e'}  +\frac{4M^2}{\alpha\gamma m}. 
\end{align}
Then from the strong convexity we have
\begin{align}\label{eq:con_reg_bar}
\| \hat x_{\e'} - x^*\|_2 
  &\leq  \sqrt{\frac{2}{\gamma}\left( \sqrt{\frac{2M^2}{\gamma }\e'} +\frac{4M^2}{\alpha \gamma m}\right)}. 
\end{align}
Equating \eqref{eq_to_prove_2} to $\e$, we get the expressions for the sample size $m$ and auxiliary precision $\e'$. Substituting both of these expressions in  \eqref{eq:con_reg_bar} we finish the proof.

\end{proof}

\section{Non-Strongly Convex Optimization Problem}
Now we  consider non-strongly convex    optimization problem 
\begin{equation}\label{eq:gener_risk_min_nonconv}
\min_{x\in X \subseteq \mathbb{R}^n} F(x) \triangleq \E f(x,\xi),
\end{equation}
where $f(x,\xi)$ is Lipschitz continuous  in $x$. Let us define   $ x^* = \arg\min\limits_{x\in X} {F}(x)$.

\subsection{The SA  Approach: Stochastic Mirror Descent}
We consider stochastic mirror descent (MD) with inexact oracle  \cite{nemirovski2009robust,juditsky2012first-order,gasnikov2016gradient-free}.\footnote{By using dual averaging scheme~\cite{nesterov2009primal-dual} we can rewrite Alg.~\ref{Alg:OnlineMD} in online regime \cite{hazan2016introduction,orabona2019modern} without including $N$ in the stepsize policy. Note, that mirror descent and dual averaging scheme are very close to each other \cite{juditsky2019unifying}.}  For a prox-function $d(x)$ and the corresponding Bregman divergence $B_d(x,x^1)$, the proximal mirror descent  step is 
\begin{equation}\label{eq:prox_mirr_step}
   x^{k+1} = \arg\min_{x\in X}\left( \eta \left\langle g_\delta(x^k,\xi^k), x\right\rangle + B_d(x,x^k)\right).
\end{equation}
We consider the simplex setup: 
 prox-function $d(x) = \la x,\log x \ra$. Here and below, functions such as $\log$ or $\exp$ are always applied element-wise. The corresponding Bregman divergence is given by the Kullback--Leibler divergence
    \[
{\rm KL}(x,x^1) = \la x, \log(x/x^1)\ra - \boldsymbol{1}^\top(x-x^1).
\]
Then the starting point is taken as $x^1 = \arg\min\limits_{x\in \Delta_n}d(x)= (1/n,...,1/n)$.

\begin{theorem}\label{Th:MDgener}
Let $ R^2 \triangleq {\rm KL}(x^*,x^1) \leq \log n $ and $D =\max\limits_{x',x''\in \Delta_n}\|x'-x''\|_1 = 2$. Let  $f:X\times \Xi \rightarrow \R^n$ be $M_\infty$-Lipschitz w.r.t. $x$ in the $\ell_1$-norm. Let $\breve x^N \triangleq \frac{1}{N}\sum_{k=1}^{N}x^k $ be the average of outputs generated by iterative formula \eqref{eq:prox_mirr_step} with $\eta = \frac{\sqrt{2} R}{M_\infty\sqrt{N} }$. Then,  with probability  
at least $1-\alpha$  we have 
\begin{equation*}
F(\breve x^N) - F(x^*)   \leq\frac{M_\infty (3R+2D \sqrt{\log (\alpha^{-1})})}{\sqrt{2N}} +\delta D = O\left(\frac{M_\infty \sqrt{\log ({n}/{\alpha})}}{\sqrt{N}} +2 \delta \right). 
\end{equation*} 
\end{theorem}

\begin{proof}
For  MD with prox-function function $d(x) = \la x\log x\ra $ the following holds for any $x\in \Delta_n$ \citep[Eq. 5.13]{juditsky2012first-order}
\begin{align*}
   \eta\la g_\delta (x^k,\xi^k), x^k -x \ra
 &\leq {\rm {\rm KL}}(x,x^k) - {\rm KL}(x,x^{k+1}) +\frac{\eta^2}{2}\|g_\delta(x^k,\xi^k)\|^2_\infty \notag\\
 &\leq {\rm {\rm KL}}(x,x^k) -{\rm {\rm KL}}(x,x^{k+1}) +\eta^2M_\infty^2.
 \end{align*}
Then by adding and subtracting the terms $\la F(x), x-x^k\ra$ and $\la \nabla  f(x, \xi^k), x-x^k\ra$  in this inequality, we get using 
Cauchy--Schwarz inequality  the following
\begin{align}\label{str_conv_W2}
   \eta\la \nabla F(x^k), x^k-x\ra   
   &\leq  \eta\la  \nabla  f(x^k,\xi^k)-g_\delta(x^k,\xi^k), x^k-x\ra \notag \\
   &+ \eta\la\nabla F(x^k)- \nabla f(x^k,\xi^k)  , x^k-x\ra + {\rm KL}(x,x^k) - {\rm KL}(x,x^{k+1}) +\eta^2M_\infty^2 \notag\\
   &\leq  \eta\delta\max_{k=1,...,N}\|x^k-x\|_1 + \eta\la\nabla F(x^k)- \nabla f(x^k,\xi^k)  , x^k-x\ra \notag \\ 
   &+{\rm KL}(x,x^k) - {\rm KL}(x,x^{k+1}) +\eta^2M_\infty^2.
\end{align}
Then using convexity of $F(x^k)$ we have
\[
 F(x^k) - F(x)\leq \eta\la \nabla F(x^k), x^k-x\ra 
\]
Then we use this for \eqref{str_conv_W2} and sum  for $k=1,...,N$ at $x=x^*$
\begin{align}\label{eq:defFxx}
     \eta\sum_{k=1}^N  F(x^k) - F(x^*) &\leq 
    \eta\delta N \max_{k=1,...,N}\|x^k-x^*\|_1
     +\eta\sum_{k=1}^N\la\nabla F(x^k)- \nabla f(x^k,\xi^k)  , x^k-x^*\ra \notag\\
    &+ {\rm KL}(x^*,x^1) - {\rm KL}(x^*,x^{N+1}) + \eta^2M_\infty^2N \notag\\ 
     &\leq  \eta\delta N{D}
     +\eta \sum_{k=1}^N\la\nabla  F(x^k)- \nabla f(x^k,\xi^k)  , x^k-x^*\ra + R^2+ \eta^2M_\infty^2N. 
\end{align}
Where we used ${\rm KL}(x^*,x^1) \leq R^2 $ and $\max\limits_{k=1,...,N}\|p^k-p^*\|_1 \leq D$.
Then using convexity of $F(x^k)$ and the definition of output $\breve x^N$ in \eqref{eq:defFxx} we have
\begin{align}\label{eq:F123xxF}
    F(\breve x^N) -F(x^*)
    &\leq  \delta D +\frac{1}{N}\sum_{k=1}^N \la\nabla F(x^k)- \nabla  f(x^k,\xi^k)  , x^k-x^*\ra + \frac{R^2}{\eta N}+ \eta M_\infty^2.
\end{align}
Next we use the {Azuma--}Hoeffding's {\cite{jud08}} inequality and get for all $\beta \geq 0$
\begin{equation}\label{eq:AzumaH}
    \mathbb{P}\left(\sum_{k=1}^{N+1}\la \nabla F(x^k) -\nabla f(x^k,\xi^k), x^k-x^*\ra \leq \beta \right)\geq 1 - \exp\left( -\frac{2\beta^2}{N(2M_\infty D)^2}\right)= 1 - \alpha.
\end{equation}
Here we used that $\la \nabla F(p^k) -\nabla f(x^k,\xi^k), x^*-x^k\ra$ is a martingale-difference and 
\begin{align*}
{\left|\la \nabla  F(x^k) -\nabla f(x^k,\xi^k), x^*-x^k\ra \right|} &\leq \| \nabla F(x^k) -\nabla W(p^k,q^k)\|_{\infty} \|x^*-x^k\|_1 \notag \\
&\leq 2M_\infty \max\limits_{k=1,...,N}\|x^k-x^*\|_1 \leq 2M_\infty D.
\end{align*}
Thus, using \eqref{eq:AzumaH} for \eqref{eq:F123xxF} we have that  with probability at least $1-\alpha$ 
\begin{equation}\label{eq:eta}
   F(\breve x^N) -  F(x^*) \leq  \delta D +\frac{\beta}{N}+ \frac{R^2}{\eta N}+ \eta M_\infty^2.
\end{equation}
Then, expressing $\beta$ through $\alpha$  and substituting $\eta = \frac{ R}{M_\infty} \sqrt{\frac{2}{N}}$ to \eqref{eq:eta}  ( such $\eta$ minimize the r.h.s. of \eqref{eq:eta}),  we get \begin{align*}
  & F(\breve x^N) -  F(x^*) \leq   \delta D + \frac{M_\infty D\sqrt{2\log(1/\alpha)} }{\sqrt{N} } + \frac{M_\infty R}{\sqrt{2N}}+ \frac{M_\infty R\sqrt{2}}{\sqrt{N}} \notag \\
  &\leq \delta D + \frac{M_\infty (3R+2D \sqrt{\log(1/\alpha) })}{\sqrt{2N}}. 
\end{align*} 
Using $R=\sqrt{\log n}$ and {$D = 2$}  in this inequality,
we obtain
\begin{align}\label{eq:final_est}
  F(\breve x^N) -  F(x^*)  &\leq \frac{M_\infty (3\sqrt{\log{n}} +4 \sqrt{\log(1/\alpha)})}{\sqrt{2N}} +2\delta. 
\end{align} 
We raise this to the second power, use that for all $a,b\geq 0, ~ 2\sqrt{ab}\leq a+b$ and then extract the square root. We obtain the following
\begin{align*}
\sqrt{\left(3\sqrt{\log{n}} +4 \sqrt{\log(1/\alpha)}\right)^2} &= \sqrt{ 9\log{n} + 16\log(1/\alpha) +24\sqrt{\log{n}}\sqrt{\log(1/\alpha)}  } \\
&\leq  \sqrt{ 18\log{n} + 32\log(1/\alpha) }.
\end{align*}
Using this for \eqref{eq:final_est}, we get the  statement of the theorem
\begin{align*}\label{eq:final_est2}
 F(\breve x^N) -  F(x^*)  &\leq \frac{M_\infty \sqrt{18\log{n} +32 \log(1/\alpha)}}{\sqrt{2N}} +2\delta  = O\left(\frac{M_\infty \sqrt{\log ({n}/{\alpha})}}{\sqrt{N}} +2\delta  \right) . 
\end{align*} 
\end{proof}

\subsection{ Penalization in the SAA Approach}
In this section, we study the SAA approach for non-strongly convex problem \eqref{eq:gener_risk_min_nonconv}. We regularize this problem by 1-strongly convex w.r.t. $x$ penalty function $r(x,x^1)$ in the $\ell_2$-norm 
 \begin{equation}\label{def:gener_reg_prob}
\min_{x\in X \subseteq \mathbb{R}^n} F_\lm(x) \triangleq \E f(x,\xi) +   \lm r(x,x^1)
\end{equation}
and we prove  that the sample sizes in the SA and the SAA approaches will be equal   up to logarithmic terms. 
The empirical counterpart of problem \eqref{def:gener_reg_prob} is
 \begin{equation}\label{eq:gen_prob_empir}
 \min_{x\in X }\hat{F}_\lm(x) \triangleq \frac{1}{m}\sum_{i=1}^m f(x,\xi_i) +   \lm r(x,x^1).
\end{equation}
Let us define $ \hat x_\lm = \arg\min\limits_{x\in X} \hat{F}_{\lm}(x)$.
The next lemma  proves the statement from \cite{shalev2009stochastic}  on boundness of the population sub-optimality in terms of the square root of empirical sub-optimality.
\begin{lemma}\label{Lm:pop_sub_opt}
Let $f(x,\xi)$ be convex and $M$-Lipschitz continuous w.r.t $\ell_2$-norm.  
Then for any $x \in X$ with probability at least $ 1-\delta$ the following holds
\[F_\lm(x) - F_\lm(x^*_\lm) \leq \sqrt{\frac{2M_\lm^2}{\lm} \left(\hat F_\lm(x) - \hat F_\lm(\hat x_\lm)\right)} + \frac{4M_\lm^2}{\alpha \lm m},\]
 where $ x^*_\lm = \arg\min\limits_{x\in X} {F}_\lm(x)$,  $M_\lm \triangleq M +\lambda \mathcal {R}^2$ and $\mathcal{R}^2 =   r(x^*,x^1)$.
\end{lemma}

\begin{proof}
Let us define $f_\lm(x, \xi) \triangleq f (x, \xi) + \lm r(x,x^1) $. As $f(x, \xi)$ is $M$-Lipschitz continuous, $f_\lm(x, \xi)$ is also Lipschitz continuous with $M_\lm \triangleq M +\lambda \mathcal {R}^2$.
From  
 Jensen's inequality for the expectation, and the module as a  convex function,  we get that 
 $F_\lm(x)$  is also  $M_\lm$-Lipschitz continuous
\begin{equation}\label{eq:MLipscht_cont}
|F_\lm(x)-  F_\lm(\hat x_\lm) |  \leq  M_\lm\| x -\hat x_\lm\|_2, \qquad \forall x \in X.
\end{equation}
From   $\lm$-strong convexity of $f(x, \xi)$, we obtain that  $\hat F_\lm(x)$ is also $\lm$-strongly convex 
\[
\|x-\hat x_\lm\|_2^2\leq \frac{2}{\lm}\left( \hat F_\lm(x)-\hat F_\lm(\hat x_\lm)  \right), \qquad \forall x \in X. 
\]
From this and \eqref{eq:MLipscht_cont}  it follows
\begin{equation}\label{eq:sup_opt_empr}
  F_\lm(x)-  F_\lm(\hat x_\lm)  \leq \sqrt{\frac{2M_\lm^2}{\lm}\left( \hat F_\lm(x)-\hat F_\lm(\hat x_\lm) \right)}.  
\end{equation}
For any $x \in X$ and $ x^*_\lm = \arg\min\limits_{x\in X} {F}_\lm(x)$ we consider
\begin{equation}\label{eq:Fxhatx}
    F_\lm(x)-  F_\lm( x^*_\lm) = F_\lm(x)- F_\lm(\hat x_\lm) + F_\lm(\hat x_\lm) - F_\lm( x^*_\lm).
\end{equation}
From \citep[Theorem 6]{shalev2009stochastic} we have with probability at least $ 1 -\alpha$
\[ F_\lm(\hat x_\lm) - F_\lm(x^*_\lm) \leq \frac{4M_\lm^2}{\alpha \lm m}.\]
Using  this and \eqref{eq:sup_opt_empr} for \eqref{eq:Fxhatx} we obtain with probability at least $ 1 -\alpha$
\[F_\lm(x) - F_\lm(x^*_\lm) \leq \sqrt{\frac{2M_\lm^2}{\lm}\left( \hat F_\lm(x)-\hat F_\lm(\hat x_\lm) \right)} + \frac{4M_\lm^2}{\alpha \lm m}.\]
\end{proof}
The next theorem proves the eliminating the linear dependence on $n$ in the sample size of the regularized SAA approach for a non-strongly convex objective 
(see estimate \eqref{eq:SNSm}), and estimates the  auxiliary precision for the regularized SAA problem \eqref{eq:aux_e_quad}.
\begin{theorem}\label{th_reg_ERM} 
Let $f(x,\xi)$ be convex and $M$-Lipschitz continuous w.r.t $x$ and
let $\hat x_{\e'}$ be such that
\[
\frac{1}{m}\sum_{i=1}^m f(\hat x_{\e'},\xi_i) +   \lm r(\hat x_{\e'}, x^1) - \arg\min_{x\in X}  \left\{\frac{1}{m}\sum_{i=1}^m f(x,\xi_i) +   \lm r(x,x^1)\right\} \leq \e'. 
\]
To satisfy
\[F(\hat x_{\e'}) -  F(x^*)\leq \e\]
with probability  at least $1-\alpha$
, we need to take  $\lm = \e/(2\mathcal{R}^2)$,
 \[m = \frac{ 32 M^2\mathcal{R}^2}{\alpha \e^2}, \]
 where
$\mathcal{R}^2 =  r(x^*,x^1)$. The precision $\e'$ is defined as
\[\e' = \frac{\e^3}{64M^2 \mathcal{R}^2}.\]
\end{theorem}

\begin{proof}
From Lemma \ref{Lm:pop_sub_opt}
we get for $x=\hat x_{\e'}$
\begin{align}\label{eq:suboptimality}
F_\lm(\hat x_{\e'}) - F_\lm(x^*_\lm) &\leq \sqrt{\frac{2M_\lm^2}{\lm}\left( \hat F_\lm(\hat x_{\e'})-\hat F_\lm(\hat x_\lm ) \right)} + \frac{4M_\lm^2}{\alpha \lm m} \notag \\
&=\sqrt{\frac{2M_\lm^2}{\lm}\e'} + \frac{4M_\lm^2}{\alpha \lm m},
\end{align}
 where we used the definition of $\hat x_{\e'}$ from the statement of the this theorem.
Then we  subtract $F(x^*)$ in both sides of \eqref{eq:suboptimality} and get 
\begin{align}\label{eq:suboptimalityIm}
  F_\lm( \hat x_{\e'}) - F(x^*)  
  &\leq \sqrt{\frac{2M_\lm^2\e'}{\lambda}} +\frac{4M_\lm^2}{\alpha\lambda m} +F_\lm(x^*_\lm)-F(x^*). 
\end{align}
Then we use
\begin{align*}
    F_\lm(x^*_\lm) &\triangleq \min_{x\in X}\left\{ F(x)+\lm r(x,x^1) \right\} && \notag\\
    &\leq F(x^*) + \lm r(x^*,x^1) &&  \text{The inequality holds for any $x \in X$}, \notag\\
    &=F(x^*) +\lm \mathcal{R}^2
\end{align*}
where $\mathcal{R} = r(x^*,x^1)$.
Then from this and \eqref{eq:suboptimalityIm} and the definition of $F_\lm(\hat x_{\e'})$ in \eqref{def:gener_reg_prob}  we get 
\begin{align}\label{eq:minim+lamb}
  F( \hat x_{\e'}) -  F(x^*) &\leq \sqrt{\frac{2M_\lm^2}{\lm}\e'} + \frac{4M_\lm^2}{\alpha \lm m} - \lambda r(\hat x_{\e'}, x^1)+{\lambda}\mathcal{R}^2\notag \\
  &\leq \sqrt{\frac{2M_\lm^2\e'}{\lambda}} +\frac{4M_\lm^2}{\alpha\lambda m} +{\lambda}\mathcal{R}^2. 
\end{align}
Assuming $M \gg \lambda \mathcal{R}^2 $ and choosing $\lm =\e/ (2\mathcal{R}^2)$ in \eqref{eq:minim+lamb}, we get the
following
\begin{equation}\label{offline23}
  F( \hat x_{\e'}) - F(x^*) =   \sqrt{\frac{4M^2\mathcal R^2\e'}{\e}} +\frac{8M^2\mathcal R^2}{\alpha m \e} +\e/2.
\end{equation}
Equating the first term and the second term in the r.h.s. of \eqref{offline23} to $\e/4$ we obtain the 
 the rest  statements of the theorem including  $  F( \hat x_{\e'}) - F(x^*) \leq \e.$ 

\end{proof}

\section{Fr\'{e}chet Mean  with respect to Entropy-Regularized Optimal Transport }\label{sec:pen_bar}

In this section, we consider the problem of finding population barycenter of independent identically distributed random discrete  measures. We define the population barycenter of distribution $\PP$ with respect  to
 entropy-regularized  transport distances 
 \begin{equation}\label{def:populationWBFrech}
\min_{p\in  \Delta_n}
W_\gamma(p)\triangleq
\E_q W_\gamma(p,q), \qquad q \sim \PP.
\end{equation}

\subsection{Properties of  Entropy-Regularized Optimal Transport}
Entropic regularization of transport distances \cite{cuturi2013sinkhorn}  improves their statistical properties  \cite{klatt2020empirical,bigot2019central} and reduces their computational complexity. Entropic regularization has shown good results in generative models \cite{genevay2017learning}, 
multi-label learning \cite{frogner2015learning}, dictionary learning \cite{rolet2016fast}, image processing  \cite{cuturi2016smoothed,rabin2015convex}, neural imaging \cite{gramfort2015fast}. 

Let us firstly remind   optimal transport problem   between histograms $p,q \in \Delta_n$  with cost matrix $C\in \R_{+}^{n\times n}$
\begin{equation}\label{eq:OTproblem}
 W(p,q) \triangleq   \min_{\pi \in U(p,q)} \la C, \pi \ra,
\end{equation}
where
\[U(p,q) \triangleq\{ \pi\in \R^{n\times n}_+: \pi \one =p, \pi^T \one = q\}.\]

\begin{remark}[Connection with the $\rho$-Wasserstein distance]
When for $\rho\geq 1$, $C_{ij} =\mathtt d(x_i, x_j)^\rho$  in \eqref{eq:OTproblem}, where $\mathtt d(x_i, x_j)$ is a distance on support points $x_i, x_j$ of space $X$, then $W(p,q)^{1/\rho}$ is known as the $\rho$-Wasserstein distance on $\Delta_n$.
\end{remark}
Nevertheless, all the results of this thesis  are based only on the assumptions that the  matrix $C \in \R_+^{n\times n}$ is symmetric and non-negative. Thus, optimal transport  problem defined in \eqref{eq:OTproblem} is    a more general  than the  Wasserstein distances.

Following \cite{cuturi2013sinkhorn}, we introduce  entropy-regularized optimal transport problem  
\begin{align}\label{eq:wass_distance_regul}
W_\gamma (p,q) &\triangleq \min_{\pi \in  U(p,q)} \left\lbrace \left\langle  C,\pi\right\rangle - \gamma E(\pi)\right\rbrace,
\end{align}
 where $\gamma>0$ and $E(\pi) \triangleq -\la \pi,\log \pi \ra $ is the  entropy. Since $E(\pi)$ is 1-strongly concave on $\Delta_n$ in the  $\ell_1$-norm, the objective in \eqref{eq:wass_distance_regul} is $\gamma$-strongly convex  with respect to $\pi$ in the $\ell_1$-norm on $\Delta_n$, and hence problem \eqref{eq:wass_distance_regul} has a unique optimal solution. Moreover, $W_\gamma (p,q)$ is $\gamma$-strongly convex with respect to $p$ in the $\ell_2$-norm on $\Delta_n$ \citep[Theorem 3.4]{bigot2019data}.

One particular advantage of the entropy-regularized  optimal transport is a closed-form representation for its dual function~\cite{agueh2011barycenters,cuturi2016smoothed} defined by the Fenchel--Legendre transform of $W_\gamma(p,q)$ as a function of $p$
\begin{align}\label{eq:FenchLegdef}
     W_{\gamma, q}^*(u) &=  \max_{ p \in \Delta_n}\left\{ \la u, p \ra - W_{\gamma}(p, q) \right\} =  \gamma\left(E(q) + \left\langle q, \log (K \beta) \right\rangle  \right)\notag\\
    &= \gamma\left(-\la q,\log q\ra + \sum_{j=1}^n [q]_j \log\left( \sum_{i=1}^n \exp\left(([u]_i - C_{ji})/\gamma\right) \right)\right)
\end{align}
  where  $\beta = \exp( {u}/{\gamma}) $, \mbox{$K = \exp( {-C}/{\gamma }) $} and $[q]_j$ is $j$-th component of  vector $q$. Functions such as $\log$ or $\exp$ are always applied element-wise for vectors.
Hence, the gradient of dual function $W_{\gamma, q}^*(u)$ is also represented in  a closed-form \cite{cuturi2016smoothed}
\begin{equation*}
\nabla W^*_{\gamma,q} (u)
= \beta \odot \left(K \cdot {q}/({K \beta}) \right) \in \Delta_n,
\end{equation*}
where symbols $\odot$ and $/$ stand for the element-wise product and element-wise division respectively.
This can be also written as
\begin{align}\label{eq:cuturi_primal}
\forall l =1,...,n \qquad [\nabla W^*_{\gamma,q} (u)]_l = \sum_{j=1}^n [q]_j \frac{\exp\left(([u]_l-C_{lj})/\gamma\right)  }{\sum_{i=1}^n\exp\left(([u]_i-C_{ji})/\gamma\right)}.
\end{align}
The dual representation of $W_\gamma (p,q) $    is
\begin{align}\label{eq:dual_Was}
  W_{\gamma}(p,q) &= \min_{\pi \in U(p,q) }\sum_{i,j=1}^n\left(  C_{ij}\pi_{i,j}  + \gamma \pi_{i,j}\log \pi_{i,j} \right)  \notag \\
  &=\max_{u, \nu \in \R^n} \left\{  \la u,p\ra + \la\nu,q\ra  - \gamma\sum_{i,j=1}^n\exp\left( ([u]_i+[\nu]_j -C_{ij})/\gamma -1  \right) \right\} \\
  &=\max_{u \in \R^n}\left\{ \la u,p\ra -
  \gamma\sum_{j=1}^n [q]_j\log\left(\frac{1}{[q]_j}\sum_{i=1}^n\exp\left( ([u]_i -C_{ij})/\gamma \right)  \right)
  \right\}. \notag 
\end{align}
Any solution $\begin{pmatrix}
 u^*\\
 \nu^*
 \end{pmatrix}$ of  \eqref{eq:dual_Was}  is a subgradient 
 of $  W_{\gamma}(p,q)$  \citep[Proposition 4.6]{peyre2019computational}  \begin{equation}\label{eq:nabla_wass_lagrang}
 \nabla W_\gamma(p,q) = \begin{pmatrix}
 u^*\\
 \nu^*
 \end{pmatrix}.
\end{equation} We consider  $u^*$ and $\nu^*$ such that
  $\la u^*, \one\ra = 0$ and $\la \nu^*, \one\ra = 0$  ($u^*$ and $\nu^*$ are determined up to an additive constant).

The next theorem \cite{bigot2019data} describes  the Lipschitz continuity of $W_\gamma (p,q)$ in $p$ on probability simplex $\Delta_n$  restricted to
\[\Delta^\rho_n = \left\{p\in \Delta_n : \min_{i \in [n]}p_i \geq \rho \right\},\]
where $0<\rho<1$ is an arbitrary small constant. 
\begin{theorem}\citep[Theorem 3.4, Lemma 3.5]{bigot2019data}
\label{Prop:wass_prop}
\begin{itemize}
\item For any $q \in \Delta_n$, 
$W_\gamma (p,q)$ is $\gamma$-strongly convex w.r.t. $p$ in the $\ell_2$-norm
\item For any $q \in \Delta_n$, $p \in \Delta^\rho_n$ and $0<\rho<1$, 
$\|\nabla_p W_\gamma (p,q)\|_2 \leq M$, where 
\[M = \sqrt{\sum_{j=1}^n\left(   2\gamma\log n  +\inf_{i\in [n]}\sup_{l \in [n]} |C_{jl} -C_{il}|  -\gamma\log \rho \right)^2}. \]
\end{itemize}
\end{theorem}
We roughly take $M = O(\sqrt n\|C\|_\infty)$ since for all $i,j\in [n], C_{ij} > 0$, we get
\begin{align*}
M &\stackrel{\text{\cite{bigot2019data}}}{=}  O\left(\sqrt{\sum_{j=1}^n\left(  \inf_{i\in [n]}\sup_{l \in [n]} |C_{jl} -C_{il}|   \right)^2} \right) \\
&= O \left(\sqrt{\sum_{j=1}^n \sup_{l \in [n]} C_{jl}^2} \right)= O\left( \sqrt{n} \sup_{j,l \in [n]} C_{jl} \right)
= O\left( \sqrt{n}\sup_{j\in [n]}\sum_{l \in [n]}C_{jl} \right)= O \left( \sqrt{n}\|C\|_\infty\right).
\end{align*}
Thus, we suppose that $W_\gamma(p,q)$ and $W(p,q)$ are Lipschitz continuous with almost the same Lipschitz constant $M$ in the $\ell_2$-norm on $\Delta_n^\rho$. 
Moreover, 
by the same arguments,  
for the Lipschitz continuity in the $\ell_1$-norm: $\|\nabla_p W_\gamma (p,q)\|_\infty \leq M_\infty$,  we can roughly estimate $M_\infty = O(\|C\|_\infty)$ by taking maximum instead of the square root of the sum.

In what follows, we use Lipshitz continuity of $W_\gamma(p,q)$ and $W(p,q)$  for measures  from $\Delta_n$ keeping in mind that   adding some noise and normalizing the measures makes them belong to $\Delta_n^\rho$. We also notice that if the measures are from  the interior of $\Delta_n$ then their barycenter will be also from  the interior of $\Delta_n$. 






\subsection{The SA Approach: Stochastic Gradient Descent}



For  problem \eqref{def:populationWBFrech}, as a particular case of problem \eqref{eq:gener_risk_min}, stochastic gradient descent method can be used. From Eq. \eqref{eq:nabla_wass_lagrang}, it follows that an approximation for the gradient of $W_\gamma(p,q)$ with respect to $p$ can be calculated by Sinkhorn algorithm \cite{altschuler2017near-linear,peyre2019computational,dvurechensky2018computational}
through the computing dual variable $u$ with $\delta$-precision
 \begin{equation}\label{inexact}
 \|\nabla_p W_\gamma(p,q) - \nabla_p^\delta W_\gamma(p,q) \|_2\leq \delta, \quad \forall q\in \Delta_n. 
\end{equation}
Here denotation  $\nabla_p^\delta W_\gamma(p,q)$ means an inexact stochastic subgradient of $ W_\gamma(p,q)$ with respect to $p$.
Algorithm \ref{Alg:OnlineGD} combines stochastic gradient descent
given by iterative formula \eqref{SA:implement_simple}   for $\eta_k = \frac{1}{\gamma k}$ with Sinkhorn algorithm (Algorithm \ref{Alg:SinkhWas})
and Algorithm \ref{Alg:EuProj} making the projection onto the simplex $\Delta_n$.


\begin{algorithm}[ht!]
\caption{Sinkhorn's algorithm \cite{peyre2019computational} for calculating $ \nabla_p^\delta W_\gamma(p^k,q^k) $}
\label{Alg:SinkhWas}  
\begin{algorithmic}[1]
\Procedure{Sinkhorn}{$p,q, C, \gamma$}
\State $a^1 \gets  (1/n,...,1/n)$,  $ b^1 \gets  (1/n,...,1/n)$  
\State $K \gets \exp (-C/\gamma)$
\While{\rm not converged}
 \State $a \gets {p}/(Kb)$
\State $ b \gets {q }/(K^\top a)$
\EndWhile
 \State \textbf{return} $  \gamma \log(a)$\Comment{Sinkhorn scaling $  a = e^{u/\gamma}$} 
  \EndProcedure
\end{algorithmic}
\end{algorithm}

\begin{algorithm}[ht!]
\caption{Euclidean Projection $\Pi_{\Delta_n}(p) = \arg\min\limits_{v\in \Delta_n}\|p-v\|_2$ onto  Simplex   $\Delta_n$ \cite{duchi2008efficient}}
\label{Alg:EuProj}  
\begin{algorithmic}[1]
\Procedure{Projection}{$w\in \R^n$}
\State Sort components of $w$ in decreasing manner: $r_1\geq r_2 \geq ... \geq r_n$.
 \State Find  $\rho = \max\left\{ j \in [n]: r_j - \frac{1}{j}\left(\sum^{j}_{i=1}r_i - 1\right) \right\}$
\State Define
    $\theta = \frac{1}{\rho}(\sum^{\rho}_{i=1}r_i - 1)$
    \State For all $i \in [n]$, define $p_i = \max\{w_i - \theta, 0\}$.
 \State \textbf{return} $p \in \Delta_n$ 
  \EndProcedure
\end{algorithmic}
\end{algorithm}

\begin{algorithm}[ht!]
\caption{Projected  Online Stochastic Gradient Descent for WB (PSGDWB)} 
\label{Alg:OnlineGD}   
\begin{algorithmic}[1]
  \Require starting point $p^1 \in \Delta_n$, realization $q^1$,  $\delta$, $\gamma$.
                \For{$k= 1,2,3,\dots$}
                \State  $\eta_{k} = \frac{1}{\gamma k}$
                \State  $\nabla_p^\delta W_\gamma(p^k,q^k)  \gets$ \textsc{Sinkhorn}$(p^k,q^k, C, \gamma)$ or the accelerated Sinkhorn \cite{guminov2019accelerated}
                \State $p^{(k+1)/2} \gets p^k - \eta_{k} \nabla_p^\delta W_\gamma(p^k,q^k)$
                \State $p^{k+1} \gets$ \textsc{Projection}$(p^{(k+1)/2})$
                  \State Sample $q^{k+1}$ 
                  \EndFor
    \Ensure $p^1,p^2, p^3...$
\end{algorithmic}
 \end{algorithm}


 For Algorithm \ref{Alg:OnlineGD} and problem \eqref{def:populationWBFrech},  Theorem \ref{Th:contract_gener} can be specified as follows 
\begin{theorem}\label{Th:contract}
Let $\tilde p^N \triangleq \frac{1}{N}\sum_{k=1}^{N}p^k $ be the average of  $N$ online outputs of Algorithm  \ref{Alg:OnlineGD} run with $\delta$. Then,  with probability  
at least $1-\alpha$  the following holds
\begin{equation*}
W_{\gamma}(\tilde p^N) - W_{\gamma}(p^*_{\gamma})
 = O\left(\frac{M^2\log(N/\alpha)}{\gamma N} + \delta  \right),  
\end{equation*}
where $p^*_{\gamma} \triangleq \arg\min\limits_{p\in  \Delta_n}
W_\gamma(p)$.

Let  Algorithm \ref{Alg:OnlineGD} run with  $\delta = O\left(\e\right)$ and  $
N =\widetilde O \left( \frac{M^2}{\gamma \e} \right) = \widetilde O \left( \frac{n\|C\|_\infty^2}{\gamma \e} \right)
$. Then,  with probability  
at least $1-\alpha$  
\begin{equation*}
 W_{\gamma}(\tilde p^N) -W_{\gamma}(p^*_{\gamma}) \leq \e \quad \text{and} \quad \|\tilde p^N - p^*_{\gamma}\|_2 \leq \sqrt{2\e/\gamma}.
\end{equation*}
The total complexity of 
 Algorithm  \ref{Alg:OnlineGD} 
is
\begin{align*}
   \widetilde O\left(  \frac{n^3\|C\|_\infty^2}{\gamma\e}\min\left\{ \exp\left( \frac{\|C\|_{\infty}}{\gamma} \right) \left( \frac{\|C\|_{\infty}}{\gamma} + \log\left(\frac{\|C\|_{\infty}}{\kappa \e^2} \right) \right),  \sqrt{\frac{n \|C\|^2_{\infty}}{ \kappa\gamma \e^2}} \right\} \right),
\end{align*}
where $ \kappa \triangleq \lm^+_{\min}\left(\nabla^2   W_{\gamma, q}^*(u^*)\right)$. 

\end{theorem}
\begin{proof}
We estimate the co-domain (image) of $W_\gamma(p,q)$
\begin{align*}
\max_{p,q \in \Delta_n}  W_\gamma (p,q) 
&= \max_{p,q \in \Delta_n} \min_{ \substack{\pi \in \R^{n\times n}_+, \\ \pi \one =p, \\ \pi^T \one = q}} ~ \sum_{i,j=1}^n (C_{ij}\pi_{ij}+\gamma\pi_{ij}\log \pi_{ij})\notag\\
&\leq \max_{\substack{\pi \in \R^{n\times n}_+, \\ \sum_{i,j=1}^n \pi_{ij}=1}}  \sum_{i,j=1}^n(C_{ij}\pi_{ij}+\gamma\pi_{ij}\log \pi_{ij}) \leq \|C\|_\infty.
\end{align*}
Therefore, $W_\gamma(p,q): \Delta_n\times \Delta_n\rightarrow \left[-2\gamma\log n, \|C\|_\infty\right]$.
Then we apply Theorem \ref{Th:contract_gener} with $B =\|C\|_\infty $  and $D =\max\limits_{p',p''\in \Delta_n}\|p'-p''\|_2 = \sqrt{2}$, and we sharply get
\begin{equation*}
W_{\gamma}(\tilde p^N) - W_{\gamma}(p^*_{\gamma})
 = O\left(\frac{M^2\log(N/\alpha)}{\gamma N} +  \delta  \right),  
\end{equation*}
Equating each terms in the r.h.s. of this equality to $\e/2$ and using $M=O(\sqrt n \|C\|_\infty)$,  we get the expressions for $N$ and $\delta$. The statement 
$\|\tilde p^N - p^*_{\gamma}\|_2 \leq \sqrt{2\e/\gamma}$
follows directly from strong convexity of $W_\gamma(p,q)$ and  $W_\gamma(p)$.

The proof of algorithm complexity  follows from the complexity
of the  Sinkhorn's algorithm.
To state the complexity of the Sinkhorn's \avg{algorithm} we firstly define \avg{$\tilde\delta$ as the accuracy in function value of the inexact solution $u$ of  maximization problem in  \eqref{eq:dual_Was}.}
Using this 
we formulate the number of iteration of the Sinkhorn's 
\cite{franklin1989scaling,carlier2021linear,kroshnin2019complexity,stonyakin2019gradient} 
\begin{align}\label{eq:sink}
 \widetilde O \left( \exp\left( \frac{\|C\|_{\infty}}{\gamma} \right) \left( \frac{\|C\|_{\infty}}{\gamma} + \log\left(\frac{\|C\|_{\infty}}{\tilde \delta} \right) \right)\right). 
\end{align}
The number of iteration for the accelerated Sinkhorn's can be improved \cite{guminov2019accelerated}
\begin{equation}\label{eq:accel}
\widetilde{O} \left(\sqrt{\frac{n \|C\|^2_\infty}{\gamma \e'}} \right). 
\end{equation}
Here $\e'$ is the accuracy in the function value, which is the expression 
$ \la u,p\ra + \la\nu,q\ra  - \gamma\sum_{i,j=1}^n\exp\left( {(-C_{ji}+u_i+\nu_j)}/{\gamma} -1  \right)$ under the maximum in \eqref{eq:dual_Was}.
From strong convexity of this objective on the space orthogonal to eigenvector $\boldsymbol 1_n$ corresponds to the eigenvalue $0$ for this function, it follows that 
\begin{equation}\label{eq:str_k}
  \e'\geq \frac{\gamma}{2}\|u - u^*\|^2_2 = \frac{\kappa}{2}\delta,  
\end{equation}
 where $\kappa \triangleq \lm^+_{\min}\left(\nabla^2   W_{\gamma, q}^*(u^*)\right)$.  From \citep[Proposition A.2.]{bigot2019data}, for the eigenvalue of $\nabla^2 W^*_{\gamma,q}(u^*)$ it holds that $0=\lm_n\left(\nabla^2   W_{\gamma, q}^*(u^*)\right) < \lm_k\left(\nabla^2   W_{\gamma, q}^*(u^*)\right)  \text{ for all } k=1,...,n-1$. Inequality \eqref{eq:str_k}   holds due to  $\nabla^\delta_p W_\gamma(p,q) := u$ in Algorithm \ref{Alg:OnlineGD} and $\nabla_p W_\gamma(p,q) \triangleq u^*$ in \eqref{eq:nabla_wass_lagrang}. 
Multiplying both of estimates \eqref{eq:sink} and \eqref{eq:accel} by
the complexity of each iteration of the (accelerated) Sinkhorn's algorithm $\avg{O}(n^2)$ and 
\ag{the} number of  iterations $
N =\widetilde O \left( \frac{M^2}{\gamma \e} \right)$ \ag{(measures)} of Algorithm \ref{Alg:OnlineGD},
and
taking the minimum,  we get the last statement of the theorem.
\end{proof}

 Next, we study the practical convergence of projected stochastic gradient descent (Algorithm  \ref{Alg:OnlineGD}).
 Using the fact that the true Wasserstein barycenter of one-dimensional Gaussian measures 
has closed form expression for the mean and the variance \cite{delon2020wasserstein}, we study the convergence to the true barycenter of
 the generated  truncated Gaussian measures.  Figure \ref{fig:gausbarSGD} illustrates the convergence in the $2$-Wasserstein distance within 40 seconds.
\begin{figure}[ht!]
\centering
\includegraphics[width=0.45\textwidth]{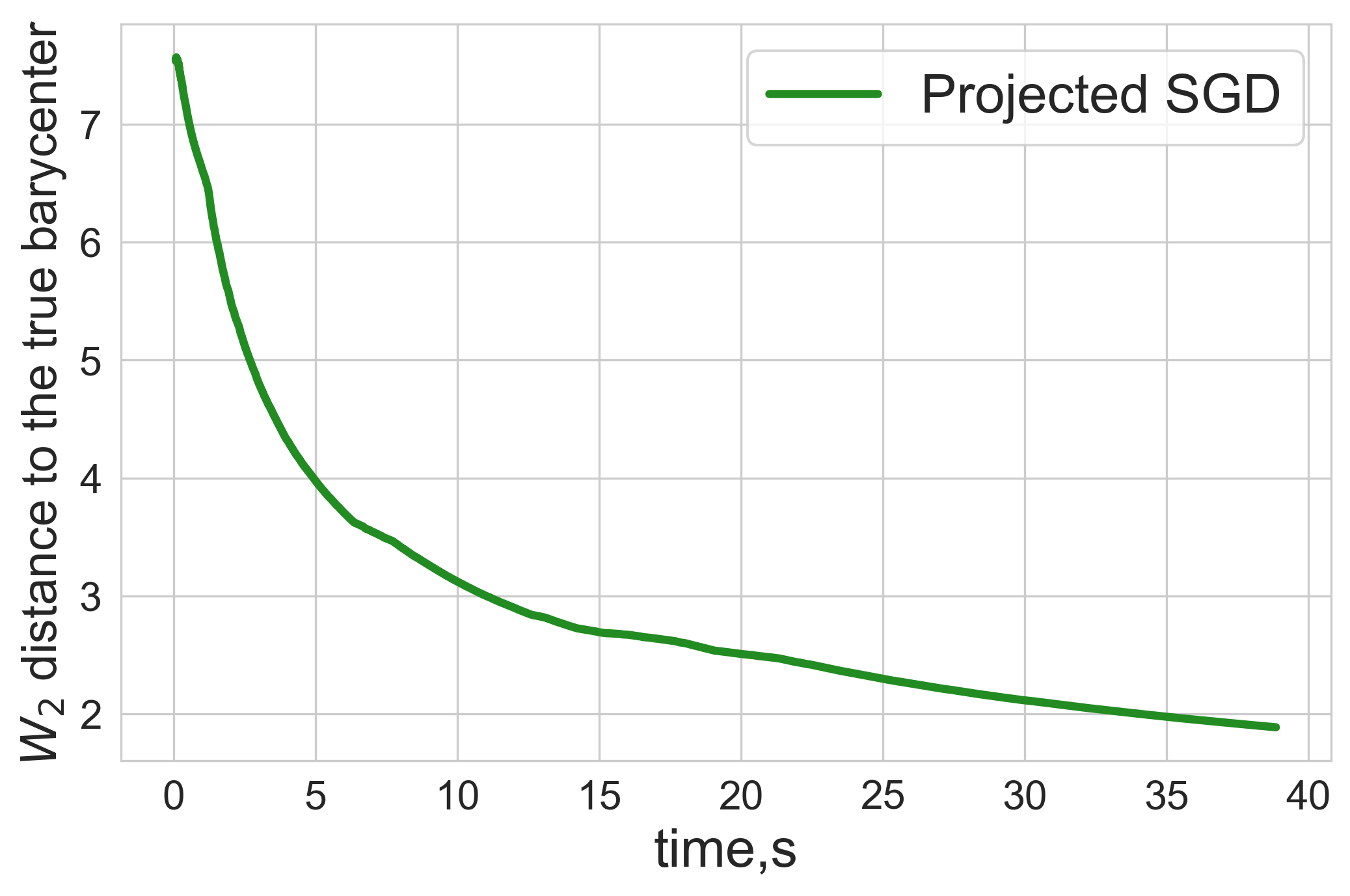}
\caption{Convergence of projected stochastic gradient descent to the true barycenter of $2\times10^4$ Gaussian measures in the $2$-Wasserstein distance. }
\label{fig:gausbarSGD}
\end{figure}

\subsection{The SAA Approach }

The empirical counterpart of problem \eqref{def:populationWBFrech} is the (empirical)  Wasserstein barycenter problem  
\begin{equation}\label{EWB_unregFrech}
\min_{p\in \Delta_n} \frac{1}{m}\sum_{i=1}^m W_\gamma(p,q_i),
\end{equation}
where  $q_1, q_2,...,q_m$ are some realizations of random variable  with distribution $\mathbb P$.

Let us define  $\hat p_\gamma^m \triangleq \arg  \min\limits_{p\in \Delta_n}{\frac{1}{m}}\sum_{i=1}^m W_\gamma(p,q_i)$  and its $\e'$-approximation $\hat p_{\e'}$ such that 
\begin{equation}\label{eq:fidelity_wass}
\frac{1}{m} \sum_{i=1}^m W_{\gamma}( \hat p_{\e'}, q_i) - \frac{1}{m} \sum_{i=1}^m W_{\gamma}(\hat p^m_{\gamma}, q_i)  \leq \e'.
\end{equation}
For instance, $\hat p_{\e'}$ can be calculated by the IBP algorithm \cite{benamou2015iterative} or the accelerated IBP algorithm \cite{guminov2019accelerated}.
The next theorem  specifies Theorem \ref{Th:contractSAA} for the Wassertein barycenter problem \eqref{EWB_unregFrech}.   
 
\begin{theorem}\label{Th:contract2}
Let $\hat p_{\e'}$ satisfies \eqref{eq:fidelity_wass}.
Then,  with probability  at least $1-\alpha$  
\begin{align*}
 W_{\gamma}( \hat p_{\e'}) -  W_{\gamma}(p_{\gamma}^*)
   &\leq \sqrt{\frac{2M^2}{\gamma}\e'}  +\frac{4M^2}{\alpha\gamma m},
\end{align*}
where $p^*_{\gamma} \triangleq \arg\min\limits_{p\in  \Delta_n}
W_\gamma(p)$.
Let $\e' = O \left(\frac{\e^2\gamma}{n\|C\|_\infty^2}  \right)$ and $m = O\left( \frac{M^2}{\alpha \gamma \e} \right) =O\left( \frac{n\|C\|_\infty^2}{\alpha \gamma \e} \right)$. Then, with probability  at least $1-\alpha$  
\[W_{\gamma}( \hat p_{\e'}) -  W_{\gamma}(p_{\gamma}^*)\leq \e \quad \text{and} \quad \|\hat p_{\e'} - p^*_{\gamma}\|_2 \leq \sqrt{2\e/\gamma}.\]
The total complexity  of the accelerated IBP computing $\hat p_{\e'}$     is 
\begin{equation*}
\widetilde O\left(\frac{n^4\|C\|_\infty^4}{\alpha \gamma^2\e^2} \right).
\end{equation*}
\end{theorem}

\begin{proof}
From Theorem \ref{Th:contractSAA} we get the first statement of the theorem
\[ W_{\gamma}( \hat p_{\e'}) -  W_{\gamma}(p_{\gamma}^*)
  \leq \sqrt{\frac{2M^2}{\gamma}\e'}  +\frac{4M^2}{\alpha\gamma m}. \]
  From \cite{guminov2019accelerated}
  we have that complexity of the accelerated IBP is
  \[
  \widetilde O\left(\frac{mn^2\sqrt n\|C\|_\infty}{\sqrt{\gamma \e'}} \right).
  \]
  Substituting the expression for $m$ and the expression for $\e'$ from Theorem \ref{Th:contractSAA} 
   \[\e' = O \left(\frac{\e^2 \gamma}{M^2}  \right), \qquad m = O\left( \frac{M^2}{\alpha \gamma \e} \right)\]
  to this equation we get the final statement of the theorem and finish the proof.
\end{proof}

 Next, we study the practical convergence of the Iterative Bregman Projections on truncated Gaussian measures.
 Figure \ref{fig:gausbarSGD} illustrates the convergence of the barycenter calculated by the IBP algorithm to the true barycenter of Gaussian measures in the $2$-Wasserstein distance within 10 seconds. For  the  convergence to the true barycenter w.r.t. the $2$-Wasserstein distance in the SAA approach, we refer to 
 \cite{boissard2015distribution}, however,  considering the  convergence  in the $\ell_2$-norm  (Theorem \ref{Th:contract2}) allows to obtain better convergence rate in comparison with the bounds for the $2$-Wasserstein distance.

\begin{figure}[ht!]
\centering
\includegraphics[width=0.5\textwidth]{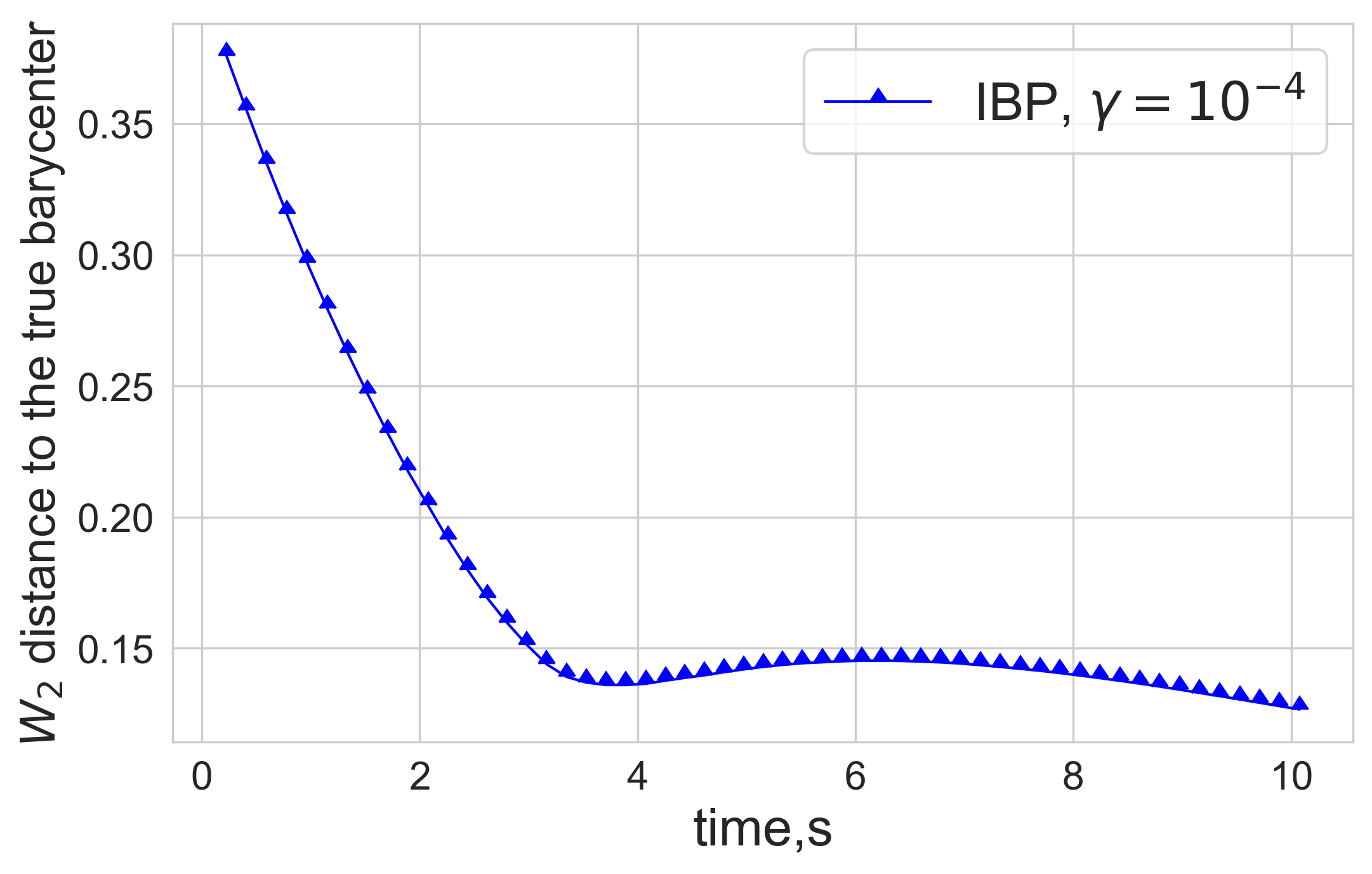}
\caption{Convergence of the Iterative Bregman Projections to the true barycenter of $2\times10^4$ Gaussian measures in the $2$-Wasserstein distance. }
\label{fig:gausbarIBP}
\end{figure}

 \subsection{Comparison of the SA and the SAA for the WB Problem}
 
Now we compare the complexity bounds for the  SA and the SAA implementations solving problem \eqref{def:populationWBFrech}.  
For the brevity, we skip the high probability details since we can fixed $\alpha$ (say $\alpha = 0.05$) in the all bounds. 
Moreover, based on \cite{shalev2009stochastic}, we assume that in fact  all  bounds of this paper have logarithmic dependence on  $\alpha$ which is hidden in  $\widetilde{O}(\cdot)$ \ddd{\cite{feldman2019high,klochkov2021stability}}.

 \begin{table}[ht!]
     \caption{Total complexity of the SA and the SAA implementations for  $ \min\limits_{p\in \Delta_n}\E_q W_\gamma(p,q)$. }
     \small
     \hspace{-0.5cm}
{ \begin{tabular}{ll}
     \toprule
   \textbf{Algorithm}  &  \textbf{Complexity}  \\
        \midrule
          Projected SGD (SA) & $     \widetilde O\left( \frac{n^3\|C\|^2_\infty}{\gamma\e} \min\left\{ \exp\left( \frac{\|C\|_{\infty}}{\gamma} \right) \left( \frac{\|C\|_{\infty}}{\gamma} + \log\left(\frac{\|C\|_{\infty}}{\kappa \e^2} \right) \right),  \sqrt{\frac{n \|C\|^2_{\infty}}{ \kappa\gamma \e^2}} \right\} \right)$ \\
         \midrule
       Accelerated IBP (SAA) &  
       $\widetilde O\left(\frac{n^4\|C\|_{\infty}^4}{\gamma^2\varepsilon^2}\right)$\\
        \bottomrule
    \end{tabular}}
    \label{Tab:entropic_OT2}
\end{table}

  Table \ref{Tab:entropic_OT2} presents the total complexity of the numerical algorithms  implementing the SA and the SAA approaches. 
When $\gamma$ is not too \ag{large}, the complexity in the first row of the table  is achieved by the second term under the minimum, namely \[\widetilde O \left(\frac{n^3\sqrt n  \|C\|^3_\infty}{\gamma \sqrt{\gamma \kappa }\e^2}\right),\]
where   $ \kappa \triangleq \lm^+_{\min}\left(\nabla^2   W_{\gamma, q}^*(u^*)\right)$. This is \dm{typically bigger than the SAA complexity when $\kappa \ll \gamma/n$.}
 Hereby, the SAA approach may outperform the SA approach provided that the regularization parameter $\gamma$ is not too large.

From the practical point of view,  the SAA implementation  converges   much faster  than the SA implementation. 
Executing the SAA  algorithm in a distributed  manner only enhances this superiority since for the case when the objective is not Lipschitz smooth, the distributed implementation of the SA approach is not possible. This is the case of the Wasserstein barycenter problem, indeed,  the objective  is Lipschitz continuous but not Lipschitz smooth.

\section{Fr\'{e}chet Mean  with respect to  Optimal Transport}\label{sec:unreg}

Now we are interested in finding a Fr\'{e}chet mean  with respect to  optimal transport
\begin{equation}\label{def:population_unregFrech}
\min_{p\in \Delta_n}W(p) \triangleq \E_q W (p,q).
\end{equation}

\subsection{The SA Approach with Regularization: Stochastic Gradient Descent}
The next theorem explains how the solution of strongly convex problem \eqref{def:populationWBFrech}    approximates
a solution of convex problem
\eqref{def:population_unregFrech} under the proper choice of the regularization parameter $\gamma$.

\begin{theorem}\label{Th:SAunreg}
Let $\tilde p^N \triangleq \frac{1}{N}\sum_{k=1}^{N}p^k $ be the average of  $N$ online outputs of Algorithm  \ref{Alg:OnlineGD} run with  $\delta = O\left(\e\right)$ and  $
N =\widetilde O \left( \frac{n\|C\|_\infty^2}{\gamma \e} \right)
$. Let $\gamma = \avg{{\e}/{(2 \mathcal{R}^2)}  } $ with $\mathcal{R}^2 = 2 \log n$. Then,  with probability  
at least $1-\alpha$  the following holds
\begin{equation*}
 W(\tilde p^N) - W(p^*) \leq \e,
\end{equation*}
 where $p^*$ is a solution of \eqref{def:population_unregFrech}.
 
 The total complexity of 
 Algorithm  \ref{Alg:OnlineGD} with the accelerated Sinkhorn
is
\[\widetilde O \left(\frac{n^3\sqrt n  \|C\|^3_\infty}{\gamma \sqrt{\gamma \kappa }\e^2}\right)= \widetilde O \left(\frac{n^3\sqrt n  \|C\|^3_\infty}{\e^3 \sqrt{\e \kappa }}\right).\]
where $ \kappa \triangleq \lm^+_{\min}\left(\nabla^2   W_{\gamma, q}^*(u^*)\right)$. 
\end{theorem}

\begin{proof}
The proof of this theorem follows from  Theorem \ref{Th:contract} and the following \cite{gasnikov2015universal,kroshnin2019complexity,peyre2019computational}
\[
 W(p) -  W(p^*) \leq W_{\gamma}(p) - W_{\gamma}(p^*) + 2\gamma\log n \leq  W_{\gamma}(p) -  W_{\gamma}(p^*_{\gamma})+ 2\gamma\log n, 
\]
where  $p \in \Delta_n $, $p^* = \arg\min\limits_{p\in \Delta_n}W(p) $. 
The choice  $\gamma = \frac{\e}{4\log n}$  ensures the following
\begin{equation*}
 W(p) -  W(p^*) \leq  W_{\gamma}(p) - W_{\gamma}(p^*_{\gamma}) + \e/2, \quad \forall p \in \Delta_n.   
\end{equation*}
This means that solving  problem \eqref{def:populationWBFrech} with $\e/2$ precision, we get a solution of problem \eqref{def:population_unregFrech} with $\e$ precision. 

 When $\gamma$ is not too \ag{large},  Algorithm  \ref{Alg:OnlineGD} uses the accelerated  Sinkhorn’s algorithm (instead of Sinkhorn’s algorithm). Thus, using $\gamma  = \frac{\e}{4\log n} $ and meaning that $\e$ is small, we get the complexity according to the statement of the theorem. 

\end{proof}

\subsection{The SA Approach: Stochastic Mirror Descent}
Now we propose an approach  to solve problem \eqref{def:population_unregFrech} without additional regularization.
The approach is based on mirror prox given by the iterative formula \eqref{eq:prox_mirr_step}.  We use simplex setup which 
provides a closed form solution for \eqref{eq:prox_mirr_step}.   Algorithm \ref{Alg:OnlineMD} presents the application of mirror prox to problem \eqref{def:population_unregFrech}, where
the gradient of $ W(p^k,q^k)$ can be calculated using dual representation of OT \cite{peyre2019computational} by any LP solver exactly
\begin{align}\label{eq:refusol}
     W(p,q) = \max_{ \substack{(u, \nu) \in \R^n\times \R^n,\\
     u_i+\nu_j \leq C_{ij}, \forall i,j \in [n]}}\left\{ \la u,p  \ra + \la \nu,q \ra  \right\}.
\end{align}
Then \[
\nabla_p    W(p,q) = u^*,
\]
where $u^*$ is a solution of \eqref{eq:refusol} such that $\la u^*,\one\ra =0$.

\begin{algorithm}[ht!]
\caption{
Stochastic Mirror Descent for the Wasserstein Barycenter Problem}
\label{Alg:OnlineMD}   
\begin{algorithmic}[1]
   \Require starting point $p^1 = (1/n,...,1/n)^T$,
   number of measures  $N$,  $q^1,...,q^N$, accuracy of gradient calculation $\delta$
   \State $\eta = \frac{\sqrt{2\log n}}{\|C\|_{\infty}\sqrt{N}}$
                \For{$k= 1,\dots, N$} 
                \State Calculate $\nabla_{p^k} W(p^k,q^k)$ solving dual LP   by any LP solver 
                \State
             \[p^{k+1} = \frac{p^{k}\odot \exp\left(-\eta\nabla_{p^k} W(p^k,q^k)\right)}{\sum_{j=1}^n [p^{k}]_j\exp\left(-\eta\left[\nabla_{p^k} W(p^k,q^k)\right]_j\right)} \]
                \EndFor
    \Ensure $\breve{p}^N = \frac{1}{N}\sum_{k=1}^{N} p^{k}$
\end{algorithmic}
 \end{algorithm}

The next theorem estimates the complexity of Algorithm \ref{Alg:OnlineMD} 

\begin{theorem}\label{Th:MD}
Let $\breve p^N$ be the output of Algorithm  \ref{Alg:OnlineMD} processing $N$ measures. Then,  with probability  
at least $1-\alpha$  we have 
\begin{equation*}
  W(\breve p^N) - W({p^*})  = O\left(\frac{\|C\|_\infty \sqrt{\log ({n}/{\alpha})}}{\sqrt{N}} \right),  
\end{equation*} 
Let  Algorithm \ref{Alg:OnlineMD} run with    $
N = \widetilde O \left( \frac{M_\infty^2R^2}{\e^2} \right) =\widetilde O \left( \frac{\|C\|_\infty^2}{\e^2} \right)
$, $ R^2 \triangleq {\rm KL}(p^1,p^*) \leq  \log n $.
Then,  with probability  
at least $1-\alpha$  
 \[  W(\breve p^N) - W(p^*) \leq \e.\]
The {total} complexity of Algorithm \ref{Alg:OnlineMD} is
\[ \widetilde O\left( \frac{ n^3 \|C\|^2_\infty}{\e^2}\right).\]
\end{theorem}

\begin{proof}
From Theorem \ref{Th:MDgener} and using $M_\infty = O\left(\|C\|_\infty\right)$, we have
\begin{align}\label{eq:final_est234}
  W(\breve p^N) -  W(p^*)  & = O\left(\frac{\|C\|_\infty \sqrt{\log ({n}/{\alpha})}}{\sqrt{N}} +2\delta  \right).
\end{align} 
Notice, that $\nabla_{p^k} W(p^k,q^k)$  can be calculated exactly by any LP solver. Thus, we take  $\delta = 0$ in \eqref{eq:final_est234} and get the first statement of the theorem.

The second statement of the theorem directly follows from this and the condition $ W(\breve p^N) - W(p^*)\leq \e$.

To get the complexity bounds we 
notice that the complexity for  calculating  $\nabla_p W(p^k,q^k)$ is $\tilde{O}(n^3)$  \cite{ahuja1993network,dadush2018friendly,dong2020study,gabow1991faster}, multiplying this by $N =  O \left( {\|C\|_\infty^2R^2}/{\e^2} \right) $ with $ R^2 \triangleq {\rm KL}(p^*,p^1) \leq \log n $, we get the last statement of the theorem.
\[\widetilde O(n^3N)  = {\widetilde O\left(n^3 \left(\frac{\|C\|_{\infty} R}{\e}\right)^2\right) =}  \widetilde O\left(n^3 \left(\frac{\|C\|_\infty}{\e}\right)^2\right).\]
\end{proof}

Next we compare  the
SA approaches with and without regularization of optimal transport in problem \eqref{def:population_unregFrech}. Entropic regularization  of optimal transport leads to strong convexity of regularized optimal transport in the $\ell_2$-norm, hence, the Euclidean setup should be used. Regularization parameter $\gamma = \frac{\e}{4 \log n}$ ensures $\e$-approximation for the unregularized solution.
In this case, we use stochastic gradient descent  with Euclidean projection onto simplex  since it converges faster for strongly convex objective. 
For non-regularized problem we can significantly use the simplex prox structure, indeed, we can apply stochastic mirror descent  with simplex setup (the Kullback-Leibler divergence as the Bregman divergence) with Lipschitz constant $M_\infty = O(\|C\|_\infty)$ that is $\sqrt{n}$ better than Lipschitz constant in the Euclidean norm $M = O(\sqrt{n}\|C\|_\infty)$.

We  studied  the convergence of stochastic mirror descent (Algorithm \ref{Alg:OnlineMD}) and stochastic gradient descent (Algorithm  \ref{Alg:OnlineGD}) in the $2$-Wasserstein distance within $10^4$ iterations (processing of $10^4$ probability measures). 
  Figure \ref{fig:gausbarcomparison1} confirms
  better convergence of stochastic mirror descent than projected stochastic gradient descent as stated in their theoretical complexity (Theorems \ref{Th:SAunreg} and \ref{Th:MD}). 
  
\begin{figure}[ht!]
\centering
\includegraphics[width=0.45\textwidth]{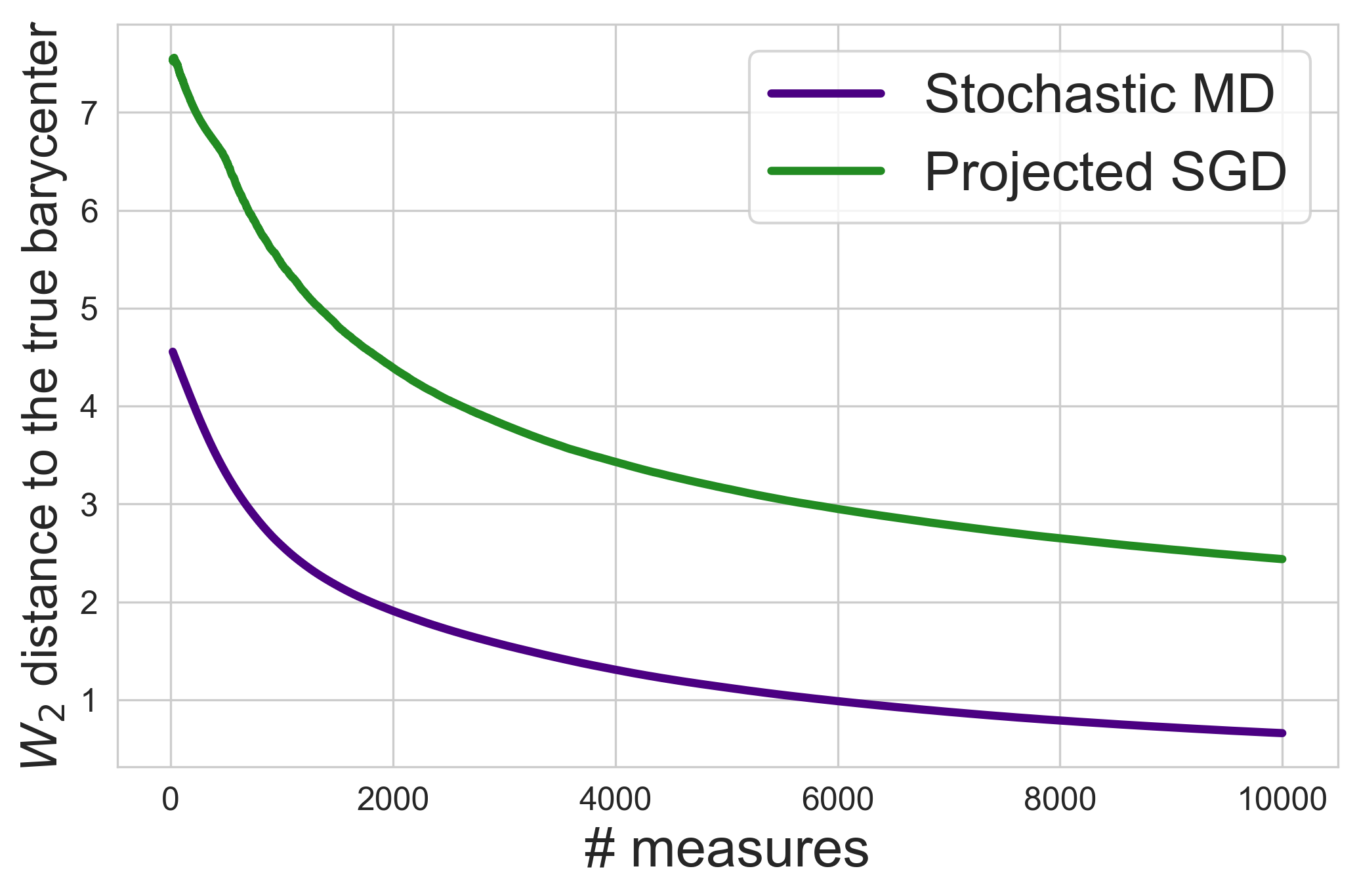}
\caption{Convergence of projected stochastic gradient descent, and stochastic mirror descent to the true barycenter of  $2\times10^4$ Gaussian measures in the $2$-Wasserstein distance. }
\label{fig:gausbarcomparison1}
\end{figure}

\subsection{The SAA Approach }

Similarly for the SA approach, we provide the proper choice of the regularization parameter $\gamma$ in  the SAA approach 
so that the 
solution of strongly convex problem \eqref{def:populationWBFrech}    approximates
a solution of convex problem
\eqref{def:population_unregFrech}.


\begin{theorem}\label{Th:SAAunreg}
Let $\hat p_{\e'}$ satisfy 
\begin{equation*}
\frac{1}{m} \sum_{i=1}^m W_{\gamma}( \hat p_{\e'}, q^i) - \frac{1}{m} \sum_{k=1}^m W_{\gamma}(\hat p^*_{\gamma}, q^i)  \leq \e',
\end{equation*}
where $ \hat p_\gamma^* = \arg\min\limits_{p\in \Delta_n}\dm{\frac{1}{m}}\sum\limits_{i=1}^m W_\gamma(p,q^i)$, $\e' = O \left(\frac{\e^2 \gamma}{n\|C\|_\infty^2}  \right)$, $m = O\left( \frac{n\|C\|_\infty^2}{\alpha \gamma \e} \right)$, and  $\gamma = {\e}/{(2 \mathcal{R}^2)} $ with $\mathcal{R}^2 = 2 \log n$.
Then,  with probability  at least $1-\alpha$  the following holds
\[W( \hat p_{\e'}) -  W(p^*)\leq \e.\]
The total complexity  of the accelerated IBP computing $\hat p_{\e'}$     is 
\begin{equation*}
\widetilde O\left(\frac{n^4\|C\|_\infty^4}{\alpha \e^4} \right). 
\end{equation*}

\end{theorem}

\begin{proof}
The proof follows from Theorem \ref{Th:contract2} and the proof of Theorem \ref{Th:SAunreg} with $\gamma = {\e}/{(4 \log n)} $.
\end{proof}

\subsection{Penalization of the WB problem}

For  the population Wasserstein barycenter problem, we construct 1-strongly convex penalty function in the $\ell_1$-norm based on Bregman divergence.
We consider the following prox-function \cite{ben-tal2001lectures}
\[d(p) = \frac{1}{2(a-1)}\|p\|_a^2, \quad a = 1 + \frac{1}{2\log n}, \qquad p\in \Delta_n\]
 that is 1-strongly convex in the $\ell_1$-norm. Then Bregman divergence $ B_d(p,p^1)$  associated with $d(p)$  is
\[B_d(p,p^1) = d(p) - d(p^1) - \la \nabla d(p^1), p - p^1 \ra.\] 
$B_d(p,p^1)$ is 1-strongly convex w.r.t. $p$ in the $\ell_1$-norm and $\tilde{O}(1)$-Lipschitz continuous in the $\ell_1$-norm on $\Delta_n$. One of the advantages of this penalization compared to the negative entropy penalization proposed in  \cite{ballu2020stochastic,bigot2019penalization}, is that we get the upper bound on the Lipschitz constant, the properties of strong convexity in the $\ell_1$-norm on $\Delta_n$ remain the same. Moreover, this penalization contributes to the better wall-clock time complexity than quadratic penalization \cite{bigot2019penalization} since the constants of Lipschitz continuity for $W(p,q)$ with respect to the $\ell_1$-norm is $\sqrt{n}$ better than with respect to the $\ell_2$-norm   but $R^2 = \|p^* - p^1\|_2^2\leq \|p^* - p^1\|_1^2 \leq  \sqrt{2} $ and $R^2_d = B_d(p^*,p^1) = O(\log n)$  are equal up to a logarithmic factor.

The regularized SAA problem is following
 \begin{equation}\label{EWB_Bregman_Reg}
 \min_{p\in \Delta_n}\left\{\frac{1}{m}\sum_{k=1}^m W(p,q^k) +  \lm B_d(p,p^1)\right\}.
\end{equation}
The next theorem is particular case of Theorem \eqref{th_reg_ERM} for the population WB problem \eqref{def:population_unregFrech} with $r(p,p^1) = B_d(p,p^1)$.
\begin{theorem}\label{Th:newreg}
Let $\hat p_{\e'}$ be  such that
 \begin{equation}\label{EWB_eps}
\frac{1}{m}\sum_{k=1}^m W(\hat p_{\e'},q^k) + \lm B_d(\hat p_{\e'},p^1)  -
\min_{p\in \Delta_n}\left\{\frac{1}{m}\sum_{k=1}^m W(p,q^k) +  \lm B_d(p,p^1)\right\}
\le \e'.
\end{equation}
To satisfy
\[ W( \hat p_{\e'}) -  W(p^*)\leq \e.\]
with probability  at least $1-\alpha$,
 we need to take  $\lm = \e/(2{R_d^2})$ and
 \[m = \widetilde O\left(\frac{ \|C\|_\infty^2}{\alpha \e^2}\right), \]
 where
$ R_d^2 =  B_d(p^*,p^1) = {O(\log n)}$. The precision $\e'$ is defined as
\[\e' = \widetilde O\left(\frac{\e^3}{\|C\|_\infty^2 }\right).\]
The total complexity  of Mirror Prox computing $\hat p_{\e'}$  is 
\[ 
  \widetilde O\left(  \frac{ n^2\sqrt n\|C\|^5_\infty}{\e^5} \right).\]
\end{theorem}
\begin{proof}
The proof is based on saddle-point reformulation of the WB problem. 
Further, we provide the explanation how to do this.
Firstly 
we   rewrite  the OT    as \cite{jambulapati2019direct}
\begin{equation}\label{eq:OTreform}
     W(p,q) = \min_{x \in \Delta_{n^2}} \max_{y\in [-1,1]^{2n}} \{d^\top x +2\|d\|_\infty(~ y^\top Ax -b^\top y)\},
\end{equation}
where  $b = 
(p^\top, q^\top)^\top$,  $d$ is vectorized cost matrix of $C$, $x$ be vectorized transport plan of  $X$,  and  $A=\{0,1\}^{2n\times n^2}$ is an incidence matrix.
Then we reformulate the WB problem as a saddle-point problem \cite{dvinskikh2020improved}
\begin{align}\label{eq:alm_distr}
         \min_{ \substack{ p \in \Delta^n, \\ \x \in \X \triangleq \underbrace{\Delta_{n^2}\times \ldots \times \Delta_{n^2}}_{m} }}    \max_{ \y \in  [-1,1]^{2mn}}
        \frac{1}{m} \left\{\boldsymbol d^\top \x +2\|d\|_\infty\left(\y^\top\boldsymbol A \x -\mathbf b^\top \y \right)\right\}, 
        \end{align}
    where
     $\x = (x_1^\top ,\ldots,x_m^\top )^\top  $,
  $\y = (y_1^\top,\ldots,y_m^\top)^\top $, 
    $\mathbf b = (p^\top, q_1^\top, ..., p^\top, q_m^\top)^\top$, 
$\boldsymbol d = (d^\top, \ldots, d^\top )^\top $, 
and    $\boldsymbol A = {\rm diag}\{A, ..., A\} \in \{0,1\}^{2mn\times mn^2}$ is  block-diagonal matrix.   
Similarly to \eqref{eq:alm_distr} we reformulate  \eqref{EWB_Bregman_Reg} as a saddle-point problem 
\begin{align*}
         \min_{ \substack{ p \in \Delta^n, \\ \x \in \X}}    \max_{ \y \in  [-1,1]^{2mn}}
         ~f_\lm(\x,p,\y) 
         &\triangleq 
        \frac{1}{m} \left\{\boldsymbol d^\top \x +2\|d\|_\infty\left(\y^\top\boldsymbol A \x -\mathbf b^\top \y \right)\right\} +\lm B_d(p,p^1)
        \end{align*}
The  gradient operator for $f(\x,p,\y)$ is  defined by
\begin{align}\label{eq:gradMPrec}
    G(\x, p, \y) = 
    \begin{pmatrix}
  \nabla_\x f \\
 \nabla_p f\\
 -\nabla_\y f 
    \end{pmatrix} = 
   \frac{1}{m} \begin{pmatrix}
   \boldsymbol d + 2\|d\|_\infty \boldsymbol A^\top \y \\
  -   2\|d\|_\infty \{[y_{i}]_{1...n}\}_{i=1}^m +\lm(\nabla d(p) - \nabla d(p^1)) \\
    2\|d\|_\infty(\boldsymbol A\x - \b)  
    \end{pmatrix},
\end{align}
where $[d(p)]_i = \frac{1}{a-1}\|p\|_a^{2-a}[p]_i^{a-1}$.

To get the complexity of MP we use the same reasons as in \cite{dvinskikh2020improved} with \eqref{eq:gradMPrec}.   The  total complexity  is
\[
\widetilde O\left(  \frac{ mn^2\sqrt{n} \|C\|_\infty}{\e'} \right)
\]
Then we use  Theorem  \ref{th_reg_ERM}
and get the exspressions for $m$, $\e'$ with $\lm = \e/(2{R_d}^2)$, where ${R_d}^2 =  B_d(p^*,p^1)$.
The number of measures is
 \[m = \frac{ 32 M_\infty^2 R_d^2}{\alpha \e^2} = \widetilde O\left(\frac{\|C\|_\infty^2}{\alpha \e^2}\right). \]
 The precision $\e'$ is defined as
\[\e' = \frac{\e^3}{64M_\infty^2 {R_d}^2} = O\left(\frac{\e^3}{\|C\|_\infty^2}\right).\]

\end{proof}

  \subsection{Comparison of the SA and the SAA for the WB Problem.}
   Now we compare the complexity bounds for the  SA and the SAA implementations solving problem  \eqref{def:population_unregFrech}.   Table \ref{Tab:OT} presents the total complexity for the numerical algorithms.

 { 
   \begin{table}[H]
     \caption{Total complexity of the SA and the SAA implementations for  $ \min\limits_{p\in \Delta_n}\E_q W(p,q)$. }
     \begin{center}
    \begin{tabular}{lll}\toprule
  \textbf{Algorithm}  & \textbf{Theorem} & \textbf{Complexity}   \\
        \midrule
          \makecell[l]{Projected SGD (SA) \\ \text{with }$\gamma = \frac{\e}{4 \log n}$}
       & \ref{Th:SAunreg} & $ \widetilde O \left(\frac{n^3\sqrt n  \|C\|^3_\infty}{\e^3 \sqrt{\e \kappa }}\right)$ \\
                  \midrule
        Stochastic MD (SA) & \ref{Th:MD} & $\widetilde O\left( \frac{n^3\|C\|^2_{\infty}}{\e^2}\right) $  \\
         \midrule
         \makecell[l]{Accelerated IBP (SAA) \\ \text{with }$\gamma = \frac{\e}{4 \log n}$}
     &  \ref{Th:SAAunreg} &
  $\widetilde{O}\left(\frac{n^4\|C\|^4_{\infty}}{\e^4}\right)$
     \\
         \midrule
     \makecell[l]{ Mirror Prox with $B_d(p^*,p^1)$\\   penalization (SAA)} & \ref{Th:newreg} & \ag{$\widetilde O\left(\frac{ n^{2}\sqrt{n}\|C\|^{5} _{\infty}}{\e^{5} }\right)$} \\
        \bottomrule
    \end{tabular}
    \label{Tab:OT}
   \end{center}
\end{table}
  }

   For the SA algorithms, which are Stochastic MD and Projected SGD,  we can conclude the following: non-regularized approach (Stochastic MD)    uses simplex prox structure and gets better complexity bounds, indeed Lipschitz constant in the $\ell_1$-norm is $M_\infty = O(\|C\|_\infty)$, whereas Lipschitz constant in the Euclidean norm is $M = O(\sqrt{n}\|C\|_\infty)$.  The practical comparison of Stochastic MD (Algorithm \ref{Alg:OnlineMD})  and Projected SGD (Algorithm \ref{Alg:OnlineGD}) can be found in  Figure \ref{fig:gausbarcomparison1}.  
   
   For the SAA approaches (Accelerated IBP and Mirror Prox with specific penalization) we enclose the following: entropy-regularized approach (Accelerated IBP) has better dependence on $\e$ than penalized approach (Mirror Prox with specific penalization), however, worse dependence   on $n$. Using Dual Extrapolation method for the WP problem from paper \cite{dvinskikh2020improved} instead of Mirror Prox allows to omit $\sqrt{n}$ in the penalized approach.

  
  One of the main advantages of the SAA approach is the  possibility to perform it in a decentralized manner 
 in contrast to the SA approach, which cannot be executed in a decentralized manner or even in distributed  or parallel fashion for non-smooth objective  \cite{gorbunov2019optimal}. This is the case of the Wasserstein barycenter problem, indeed,  the objective  is Lipschitz continuous but not Lipschitz smooth.

\section*{Acknowledgements}

The work   was 
supported by the Russian Science Foundation (project 18-71-10108), \url{https://rscf.ru/project/18-71-10108/}; and by the Ministry of Science and Higher Education of the Russian Federation (Goszadaniye) number 075-00337-20-03, project No. 0714-2020-0005.



\bibliographystyle{tfs}

\bibliography{references}

\end{document}